\documentclass{amsart} 

\usepackage{amsmath,amssymb,amsthm} 
\usepackage{graphicx} 
\usepackage[labelformat=empty]{subfig} 
\usepackage{floatrow}
\usepackage[arrow, matrix, curve]{xy} 
\usepackage{color} 
\usepackage{makecell} 
\usepackage{multirow} 
\usepackage{enumerate}
\usepackage{booktabs} 
\usepackage{siunitx} 
\usepackage{hyperref} 

\newcommand{\R}{\mathbb{R}}
\newcommand{\Z}{\mathbb{Z}}

\newcommand{\C}{\mathbb{C}}

\newcommand{\Ob}{\operatorname{Ob}}
\newcommand{\e}{\varepsilon}

\newcommand{\spin}{\ifmmode{\rm Spin}\else{${\rm spin}$\ }\fi}
\newcommand{\spinc}{\ifmmode{{\rm Spin}^c}\else{${\rm spin}^c$}\fi}

\newcommand{\MCG}{\mathrm{MCG}}

\makeatletter
\newcommand{\Vast}{\bBigg@{2.5}} 
\makeatother

\newtheoremstyle{thm}{}{}{\itshape}{}{\bfseries}{}{ }{} 
\newtheoremstyle{definition}{}{}{}{}{\bfseries}{}{ }{} 

\theoremstyle{thm}
\newtheorem{Theorem}{Theorem}[section]
\newtheorem{thm}[Theorem]{Theorem}
\newtheorem{lem}[Theorem]{Lemma}
\newtheorem{prop}[Theorem]{Proposition}
\newtheorem{cor}[Theorem]{Corollary}
\newtheorem*{Theorem-ohne}{Theorem}

\theoremstyle{definition}

\newtheorem{rem}[Theorem]{Remark}
\newtheorem{ex}[Theorem]{Example}
\newtheorem*{ack}{Acknowledgments}
\newtheorem*{convention}{Convention}

\begingroup\expandafter\expandafter\expandafter\endgroup
\expandafter\ifx\csname pdfsuppresswarningpagegroup\endcsname\relax
\else
\fi

\makeatletter
\renewcommand{\p@subfigure}{}
\makeatother

\definecolor{amaranth}{rgb}{0.9, 0.17, 0.31} 
\definecolor{carrotorange}{rgb}{0.93, 0.57, 0.13} 
\definecolor{citrine}{rgb}{0.89, 0.82, 0.04} 
\definecolor{dartmouthgreen}{rgb}{0.05, 0.5, 0.06} 
\definecolor{ballblue}{rgb}{0.13, 0.67, 0.8} 
\definecolor{ceruleanblue}{rgb}{0.16, 0.32, 0.75} 
\definecolor{amethyst}{rgb}{0.6, 0.4, 0.8} 
\definecolor{amber}{rgb}{1.0, 0.75, 0.0} 
\definecolor{burlywood}{rgb}{0.87, 0.72, 0.53} 

\numberwithin{equation}{section}

\begin{document}


\title[Trisecting a 4-dimensional book into three chapters] {Trisecting a $4$-dimensional book into three chapters} 

\author{Marc Kegel}
\address{Humboldt Universit\"at zu Berlin, Rudower Chaussee 25, 12489 Berlin, Germany \newline \indent Universität Heidelberg, Im Neuenheimer Feld 205, 69120 Heidelberg, Germany}
\email{{kegemarc@hu-berlin.de}, {kegelmarc87@gmail.com}}

\author{Felix Schm\"aschke}
\address{Humboldt-Universit\"at zu Berlin, Rudower Chaussee 25, 12489 Berlin, Germany.}
\email{schmascf@math.hu-berlin.de}


\date{\today} 

\begin{abstract}
We describe an algorithm that takes as input an open book decomposition of a closed oriented $4$-manifold and outputs an explicit trisection diagram of that $4$-manifold. Moreover, a slight variation of this algorithm also works for open books on manifolds with non-empty boundary and for $3$-manifold bundles over $S^1$. We apply this algorithm to several simple open books, demonstrate that it is compatible with various topological constructions, and argue that it generalizes and unifies several previously known constructions.
\end{abstract}


\makeatletter
\@namedef{subjclassname@2020}{%
  \textup{2020} Mathematics Subject Classification}
\makeatother


\maketitle


\section{Introduction}

In this manuscript, we study the connections of open book decompositions with trisections on $4$-manifolds. We briefly recall the necessary notions.

\subsection{Open books}

Let $W^4$ be a smooth closed oriented $4$-manifold.  An \textit{open book decomposition} of $W$ is a codimension $2$-submanifold $B^2$ with trivial normal bundle in $W$, called the \textit{binding}, and a (locally) trivial fibration $p\colon W\setminus B\rightarrow S^1$ which is near $B$ given as the standard angular fibration $\nu B\setminus B\cong B\times (D^2\setminus\{0\})\rightarrow S^1$. The closure of a preimage $\overline{p^{-1}(\varphi)}$ is a compact $3$-manifold $M$, called the \textit{page}. The boundary of a page is the binding $B$. Turning $M$ once around the $S^1$-direction gives the isotopy class of a diffeomorphism $\Phi$ of $M$ fixing the boundary of $M$ pointwise. The conjugacy class of $\Phi$ in the mapping class group is called the \textit{monodromy}. We denote by $\Ob(M,\Phi)$ the \textit{abstract open book} with page $M$ and monodromy $\Phi$. These definitions extend to manifolds with boundary and are discussed in detail in Section~\ref{sec:openbooks}.

Open books occur naturally in many topological situations~\cite{Mi68} and can be seen as a generalization of fibered knots~\cite{Ru90}. In the last decades, open books received a lot of attention due to their close relations to contact geometry~\cite{Gi02,Et06}, foliations~\cite{It72,It73} and related areas~\cite{CPV18}, to only mention a few results. While the existence question of open books is completely settled in all other dimensions~\cite{Qu79}, in dimension $4$ this remains, as so many other questions, mysterious. 

\subsection{Trisections}
On the other hand, recently the theory of \textit{trisections} of $4$-man\-i\-folds was introduced in~\cite{GK16} saying that any closed $4$-manifold can be decomposed in three $4$-dimensional $1$-handlebodies intersecting pairwise in a $3$-dimensional $1$-handlebody and with triple intersection a closed surface. If a $4$-manifold carries an open book structure it carries natural symmetries that should give rise to natural trisections of the underlying $4$-manifold. The purpose of this paper is to work out this connection.

\subsection{Main result} We say that a diffeomorphism $\Phi\colon M \to M$ preserves a Heegaard splitting $M=K\cup_F H$ when $\Phi(H) = H$ and thus also $\Phi(K) = K$. In Proposition~\ref{prop:splitting} below, we show that for any diffeomorphism $\Phi$ there exists a Heegaard splitting, which is preseved by $\Phi$. We define $g_\Phi(M)$ as the minimal genus of a Heegaard splitting of $M$ that preserves $\Phi$. Moreover, we denote by $g(F)$ the genus of a surface $F$ and by $|X|$ the number of connected components of a space $X$.

\begin{thm}\label{thm:A}
Let $W$ be a compact smooth and oriented $4$-manifold possibly with boundary.
\begin{enumerate}
   \item If $W$ is closed and admits the structure of an abstract open book $\Ob(M,\phi)$, then $W$ carries a $(g,k)$-trisection where 
    \begin{equation*}
      (g,k)= \big(3g_\Phi(M)-2g(\partial M)+2(\vert \partial M\vert-1),g_\Phi(M)\big).
    \end{equation*}
    \item \label{closedbundle} If $W$ is closed and admits the structure of a fibre bundle over $S^1$ with fibre $M$ and monodromy $\Phi$ then $W$ carries a $(g,k)$-trisection where
    \begin{equation*}
       (g,k)= \big(3g_\Phi(M)+1,g_\Phi(M)+1\big).
    \end{equation*}
    \item If $W$ has non-empty boundary and admits the structure of an abstract open book $\Ob(M,B\cup_C P,\Phi)$, then $W$ carries a relative $(g,k;p,b)$-trisection where 
    \begin{equation*}
    (g,k;p,b)=\big(3g_\Phi(M)-2g(B)+2(\vert B\vert-1),g_\Phi(M);g(P),|\partial B|\big).
    \end{equation*}
    \item If $W$ has non-empty boundary and admits the structure of fibre bundle over $S^1$ with fibre $M$ with boundary $\partial M=P$ and monodromy $\Phi$, then $W$ carries a relative $(g,k;p,b)$-trisection where
    \begin{equation*}
      (g,k;p,b)=  \big(3g_\Phi(M)+1,g_\Phi(M)+1;g(P),0\big).
    \end{equation*}
\end{enumerate}
Moreover, these results are constructive. There exist algorithms to create explicit trisection diagrams of the claimed genera from natural input data.
\end{thm}

The notation $\Ob(M,B\cup_C P,\Phi)$ for open books with boundaries will be introduced in Section~\ref{sec:openbooks}. We also remark that there is a compression body with lower genus boundary $\partial M$ and higher genus boundary given by the Heegaard surface which has genus $g_\Phi(M)$. This implies that $g(\partial M) \leq g_\Phi(M)$ and thus the above claimed trisection genera are always non-negative.  

As applications of our main result, we will discuss in Section~\ref{sec:ex} several natural open books on $4$-manifolds and explicitly construct their induced trisections diagrams with the above algorithm. In particular, we describe trisection diagrams of surface bundles over $S^2$ and $S^1$-bundles over $3$-manifolds and we discuss non-trivial open books of spun lens spaces. We will see that this yields in many cases minimal trisections. More theoretically, we show that the connected sum of open books yields the connected sum of the induced trisections and that stabilizations of open books correspond to stabilizations of the trisection diagrams. 

\subsection{Previous work}
Some special cases of these theorems were already contained in the literature. In~\cite{BS18,BS21,GK16} trisections on $S^1\times M$ are constructed and the case of more general $3$-manifold bundles over $S^1$ (Theorem~\ref{thm:A}.(\ref{closedbundle}) and~(4)) was first handled by Koenig~\cite{Ko17} and extended to the case with boundary in~\cite{Di23}. We include the bundle case here (which can be seen as 'degenerated' open books with empty binding) since this gives a new perspective on Koenig's result and is completely different from his approach.

Moreover, in~\cite{Me18}, cf.~\cite{Ha17}, trisections on spun $4$-manifolds are constructed which are in the language of open books nothing but an open book with page a $3$-manifold with a single $S^2$ boundary component and trivial monodromy. Similarly, trisections on twist spun $4$-manifolds are constructed in~\cite{Me18} which can be seen as open books of punctured $3$-manifolds with a single boundary parallel sphere twist as monodromy. Actually, our proofs extend the ideas presented by Meier in the case of spuns and twist spuns. The results obtained in this work can be seen as generalizations and unification of the above-mentioned results.

In~\cite{Mi21} conditions are presented under which an open book of $S^3$ with binding a slice knot $K$ extends to an open book of $D^4$ with binding a slice disk of $K$. It would be interesting to see what the corresponding trisection diagrams of $D^4$ look like. 
In~\cite{Is18} spun lens spaces are used to construct non-diffeomorphic trisections on the same underlying $4$-manifolds and one might wonder to which extent this might generalize to arbitrary open books.
Moreover, it is claimed in~\cite{Di23} that trisection diagrams of some special open books can also be derived with the methods from~\cite{Ko17}. But while Koenig's approach yields often minimal trisections on $3$-manifold bundles over $S^1$~\cite{Ko17}, the trisections coming from the special open books constructed in~\cite{Di23} are always non-minimal since their trisection genus is always larger than the trisections genus we get from Theorem~\ref{thm:A}.

A general method to construct trisection diagrams from a given triangulation of a $4$-manifold is presented in~\cite{BHRT18}. An algorithm to construct from a $4$-dimensional open book a Kirby diagram is presented in~\cite{Hs23} and obtaining a trisection diagram from a Kirby diagram is also possible~\cite{Ke22}. So the mere existence of an algorithm to construct a trisection diagram of an open book is already contained in the literature. The point of Theorem~\ref{thm:A} is that this algorithm works well in practice, yields trisections with small genus (which are often minimal), and is compatible with other geometric operations, see Section~\ref{sec:ex}.

\subsection{Outline and organization of the paper}
We briefly outline our arguments and first start with the idea in the $3$-dimensional case. Let $M$ be a $3$-manifold with open book decomposition $p\colon M\setminus B\rightarrow S^1$ and page a compact surface $F$ with boundary $B$. By trisecting $S^1$ in three parts $A_1=[0,2\pi/3]$, $A_2=[2\pi/3,4\pi/3]$ and $A_3=[4\pi/3,2\pi]$, we get a trisection of $M$ in three $3$-dimensional handlebodies all intersecting in a $1$-manifold as follows: $M_i=p^{-1}(A_i)$ are all diffeomorphic to $I\times F$ and thus diffeomorphic to a $3$-dimensional $1$-handlebody. They intersect pairwise in $F$, i.e.\ a $2$-dimensional handlebody and the triple intersection is the binding $B$. Not that the binding is in general not connected. But every $3$-manifold admits an open book with connected binding.

The same construction will work for a $4$-dimensional open book as well. In that case, $p^{-1}(A_i)$ will be of the form $I \times \text{page}$ and thus all be diffeomorphic, but not a $3$-dimensional $1$-handlebody in general. So in the $4$-dimensional case, additional modifications will be necessary.

In Section~\ref{sec:background} we discuss the necessary background on open books, trisections, and Heegaard splittings. In particular, we discuss results about certain mapping class groups of compression bodies adapted to Heegaard splittings that were not present in the literature and might be of independent interest (Theorems~\ref{thm:handle_slides_generate} and~\ref{thm:kernel}). These results can be simplified as follows. The monodromy $\Phi$ induces after isotopy an element in the mapping class group of the compression body $K$ relative to its lower boundary, for simplicity called $\Phi$ again. We show on the one hand that $\Phi$ is given as a product of (half-)twists along properly embedded annuli and disks, and can on the other hand also be expressed by so-called elementary handle maps (slight generalizations of $1$-handle slides). 

In Section~\ref{sec:proof} we will give a constructive proof of Theorem~\ref{thm:A}, where we explicitly construct the spine of the trisection. Here the idea is that the Heegaard splitting of the page $M$ induces a splitting of the $4$-dimensional open book into two open books (with possibly empty binding). We then construct a trisection of the part coming from the $1$-handlebody of the Heegaard splittings. By following the traces of the handle maps induced by the monodromy we get an adapted handle decomposition of this open book, from which we read off the trisection surface together with the $(\alpha,\beta,\gamma)$-curves. Note that in the case of trivial monodromy and a single $S^2$-binding, this agrees with Meier's trisection of spun $3$-manifolds~\cite{Me18}.

In practice, it is usually easy to express the monodromy via elementary handle maps and thus our algorithm works well in these examples (see Section~\ref{sec:ex}). 

\begin{convention}
We work in the smooth and oriented category. All manifolds, maps, and ancillary objects are assumed to be smooth and oriented or orientation-pre\-ser\-ving. When we consider submanifolds up to isotopy we always assume that they intersect transversely. 
\end{convention}

\begin{ack}
Part of this project was carried out at the Summer 2021 trisection workshop where we also first presented the results of this paper in the final talks. We thank the organizers of that workshop, Alexander Zupan and Jeffrey Meier, for useful discussions and feedback. We also thank the referee for the careful examination of our manuscript and the many useful hints and suggestions.

M.K.\ is partially supported by the SFB/TRR 191 \textit{Symplectic Structures in Geometry, Algebra and Dynamics}, funded by the DFG (Projektnummer 281071066 - TRR 191).
\end{ack}

\section{Background}\label{sec:background}
In this section, we recall the necessary background material for our proofs. 
\subsection{Open books} \label{sec:openbooks}

Let $W$ be a compact $n$-manifold.  An \emph{open book} decomposition of $W$ is a properly embedded (non-empty) codimension-$2$ submanifold $B$ with trivial normal bundle $\nu B$ 
together with a locally trivial fibration 
\begin{equation*}
    p\colon W\setminus B\longrightarrow S^1
\end{equation*}
that looks on a tubular neighborhood $\nu B$ of $B$ like the standard angular fibration
\begin{align*}
    p\colon \nu B\setminus B\cong B\times (D^2\setminus\{0\})&\longrightarrow S^1\\
    \big(b,(r,\varphi)\big)&\longmapsto \varphi
\end{align*}
and whose restriction to the boundary also induces a locally trivial fibration
\begin{equation*}
    p\vert_{\partial W\setminus\partial B}\colon \partial W\setminus\partial B\longrightarrow S^1.
\end{equation*}
Here $B$ is the \emph{binding} and a preimage $M_\theta =\overline{p^{-1}(\theta)}$ of a point $\theta \in S^1$ is called \emph{page}. We emphasize that the boundary of any page is the binding $B$ and thus $B$ is nullhomologous in $W$.

Abstractly, we can see an open book also as page bundle over $S^1$ (i.e.\ a mapping torus of a diffeomorphism) where we fill the boundary components by collapsing the boundaries of the pages to a single binding. Let $M$ be a compact $(n-1)$-manifold with non-empty boundary $\partial M$ together with a (possibly empty) suture $C \subset \partial M$. Here a \emph{suture} $C \subset \partial M$ is a codimension-$1$ submanifold such that the complement has two connected components denoted $\partial_- M$ and $\partial_+ M$, i.e.~$\partial M=\partial_-M\cup_C \partial_+ M$. Given a diffeomorphism $\Phi\colon M \to M$ which is the identity on $\partial_- M$, we define the \emph{mapping torus} $M_\Phi$ as 
\begin{equation*}
  M_\Phi= M \times [0,1] / (x,1) \sim (\Phi(x),0).
\end{equation*}
Then the compact $n$-manifold 
\begin{equation*}
  \operatorname{Ob}(M,\partial_-M\cup_C \partial_+ M,\Phi)= M_\Phi \cup_{\partial_- M  \times S^1} \partial_- M \times D^2
\end{equation*}
(where we glue the pieces with the obvious map preserving the product structure) carries the structure of an open book. If $\partial_- M = \partial M$ we just write $\operatorname{Ob}(M,\Phi)$. By construction the boundary of $\operatorname{Ob}(M,\partial_-M\cup_C \partial_+ M,\Phi)$ is given by the mapping torus $\big(\partial_+ M\big)_{\Phi}$ and the binding is given by $\partial_- M$.

We say that a compact $n$-manifold $W$ carries an \emph{abstract open book decomposition} if it is  diffeomorphic to $\operatorname{Ob}(M,\partial_-M\cup_C \partial_+ M,\Phi)$ for some suture $C$ of $\partial M$ and some map $\Phi\colon M\rightarrow M$ as above. The conjugacy class of $\Phi$ seen as an element in the mapping class group is called the \emph{monodromy}. It is easy to see that any open book induces an abstract open book. Sometimes it will be useful to think of open books as abstract open books. 

We refer to Section~\ref{sec:ex} for explicit examples of open books. Note that every $4$-manifold $W$ admitting an open book decomposition has vanishing signature $\sigma$, since we can decompose $W$ in two parts: the neighborhood of the binding $B\times D^2$ and the mapping torus $M_\Phi$ over $M$ with monodromy $\Phi$. Both parts admit an orientation reversing diffeomorphism and thus have vanishing signature. By Novikov additivity the signature of the whole $4$-manifold vanishes. In addition, there is a second algebraic topology obstruction for the existence of an open book decomposition, which we call the Quinn invariant $i(W)$. In~\cite{Qu79} it is shown that any manifold of dimension $n\neq4$ with vanishing signature and vanishing Quinn invariant carries an open book decomposition. In forthcoming work~\cite{KS23} we will discuss the existence question of open books on $4$-manifolds in more detail. In particular, we will show that any finitely presented group is realized as the fundamental group of a closed $4$-manifold carrying the structure of an open book, analyze the stable existence question on open books, and relate them to branched covers of $4$-manifolds. Here it seems also worth pointing out that a \textit{topological} $4$-manifold $W$ that carries an open book necessarily carries a smooth structure since this manifold is obtained by gluing together $3$-dimensional manifolds via homeomorphisms which are isotopic to diffeomorphisms.

It would be interesting to see if one can determine if a given trisection is induced by an open book. Here it is worth pointing out that we can compute the signature~\cite{FKSZ18,Ta21} and the intersection form with coefficients in the group ring~\cite{FM20} from any trisection diagram.

\subsection{Mapping class groups of compression bodies} 
In our later arguments, we will need to know certain basic properties of the mapping class groups of compression bodies. Similar results are stated in~\cite{Mc06,Su77,Wa98}. Here we just present the adaptations to our specific situation. 

A \emph{compression body $K$} is a compact connected $3$-dimensional cobordism between surfaces $\partial_- K$ and $\partial_+ K$ such that $K$ is obtained by attaching $1$-handles to $\partial_- K \times [0,1]$ at $\partial_- K \times  \{1\}$. There is a dual viewpoint by attaching  $2$-handles and $3$-handles to $\partial_+ K \times [0,1]$. We call $\partial_- K$ and $\partial_+ K$ the \emph{lower} resp.\ \emph{higher genus boundary} of $K$. Unlike at some places in the literature, the surface $\partial_- K$ is neither required to be connected, closed nor without $2$-sphere components, albeit the surfaces $\partial_\pm K$ are compact and $\partial_+ K$ is connected. Additionally, we allow $\partial_- K =\emptyset$, in which case the $1$-handles are attached to a $3$-ball, and $K$ is just an ordinary handlebody.  If $\partial_\pm K$ have non-trivial boundary then we must have $\partial\partial_- K \cong \partial \partial_+ K$ and $K$ is a sutured manifold with suture $C\cong \partial \partial_\pm K$ such that $\partial K \setminus \nu C = \partial_- K \cup \partial_+ K$.   Note that the diffeomorphism type of a compression body is completely determined by the topology of its boundary. We denote by $\MCG(K,\partial_- K)$ the group of orientation preserving diffeomorphisms of $K$ that are the identity on $\partial_-K$ (and then without loss of generality also on a neighborhood $\nu C$ of the suture $C$) modulo isotopy through such diffeomorphisms. 

We want to understand the generators of the group $\MCG(K,\partial_- K)$. In the case of a handle body, a set of generators was given in~\cite{Su77,Wa98}, cf.~\cite{He20}. We will show that similar generators also work in our setup. In order to efficiently describe these generators we first explain a way to express diffeomorphisms via twists along properly embedded surfaces, defined as follows in a general $3$-manifold $M$. Let $S\subset M$ be a properly embedded $2$-disk, annulus, $2$-torus or $2$-sphere. We identify a regular neighborhood of $S$ in $M$ with $S\times [0,1]$. We choose a representative $\varphi_t\colon S \to S$ of a generator of $\pi_1(\mathrm{Diff}(S),\mathrm{id})$. A \emph{twist along $S$} is defined to be the following diffeomorphism 
\begin{equation*}
    \Phi_S\colon S \times [0,1] \longrightarrow S \times [0,1]\,,\qquad
    (p,t) \longmapsto \big(\varphi_t(p),t\big)
\end{equation*} 
extended by the identity to all of $M$. In the case that $S$ is $D^2$ or $S^1 \times [0,1]$ seen as subsets of $\C$ the twist along $S$ is explicitly given by
\begin{equation*}
    (z,t)\longmapsto( e^{2\pi i t} z,t).
\end{equation*}
In principle, this definition makes sense also for other embedded surfaces $S$ with trivial normal bundles. However, it turns out that the group $\pi_1(\mathrm{Diff}(S),\mathrm{id})$ is trivial for all other orientable surfaces $S$~\cite{EE67}. In the case when $S$ is a $2$-torus, the group $\pi_1(\mathrm{Diff}(S),\mathrm{id})$ is canonically isomorphic to $\pi_1(S,p)$ via $[\varphi_t] \mapsto [t \mapsto \varphi_t(p)]$. In order to specify the choice of the generator along which we define the twist, it is convenient to present a closed embedded curve $c$ in $S$ and use the previous isomorphism. In that case, we say that $\Phi_S$ is the twist along $S$ \emph{in the direction of $c$}. 

In the case of a separating disk, we define a square root of a twist along this disk. More precisely, suppose that a properly embedded disk $S\cong D^2$ cuts $M$ into $M_0$ and $M_1$, such that one of these pieces, say $M_1$, admits an involution $f$ that is $-\operatorname{Id}_{D^2}$ on $S$. Then the \emph{half-twist along $S$} is defined to be the diffeomorphism
\begin{equation*}
    \Psi_S\colon S \times [0,1] \longrightarrow S \times [0,1]\,,\qquad (p,t) \longmapsto \big(e^{\pi i t} z,t\big)
\end{equation*}
extended by the identity on $M_0$ and by $f$ on $M_1$. 

\begin{ex}\label{ex:disk_annulus_twists} Let $K$ be a compression body. A diffeomorphism as in $(i)$ -- $(iv)$ below is called \emph{elementary handle map}. See Figure~\ref{fig:handle_maps} for sketches of these.
\begin{enumerate}[(i)]
    \item Let $D$ be a co-core of a $1$-handle. The disk twist along $D$ is isotopic to the full rotation of one of the feet of the $1$-handle. We call this \emph{twisting a foot} of a handle. 
 \item Let $D$ be a proper disk in $K\setminus a$ for some $1$-handle $a$, which is isotopic to a disk in $\partial_+(K\setminus a)$ containing both attaching regions of $a$. Then $D$ cuts $K$ into two pieces. The piece containing the handle $a$ admits an involution inverting an orientation of the core of $a$. Then the half-twist along $D$ interchanges the two feet of the handle. We call this map \emph{inverting the handle $a$}. 
\item Let $A$ be a proper annulus in $K\setminus a$, which is isotopic to an annulus in $\partial_+(K\setminus a)$ containing precisely one attaching region of a $1$-handle $a$.
Then the twist along $A$ is isotopic to moving the foot of the $1$-handle corresponding to the attaching region along the core of the annulus. We call this map \emph{sliding a foot} of a handle. We encode the handle slide in terms of an arc $c$ in $\partial_+ (K \setminus a)$ with boundary on one attaching region of $a$, say $D$. We retrieve $A$ as the boundary of a regular neighborhood of $c \cup D$.
\item Let $D$ be a proper disk in $K\setminus a \cup a'$ for two $1$-handles $a$ and $a'$, which is isotopic to a disk in $\partial_+(K\setminus a \cup a')$ containing both attaching regions of $a$ and $a'$. Then $D$ cuts $K$ into two pieces. The piece containing the handles $a$ and $a'$ admits an involution mapping $a$ to $a'$ orientation preserving. We call the half-twist along $D$ \emph{exchanging handles $a$ and $a'$}.
\end{enumerate} 
To verify the statements about the elementary diffeomorphism it is convenient to study the action of the diffeomorphisms on a cut system of $\partial_+ K$ and use Theorem~\ref{thm:kernel} given below. 
\end{ex}

\begin{figure}[htbp] 
\centering
\def\svgwidth{\columnwidth}
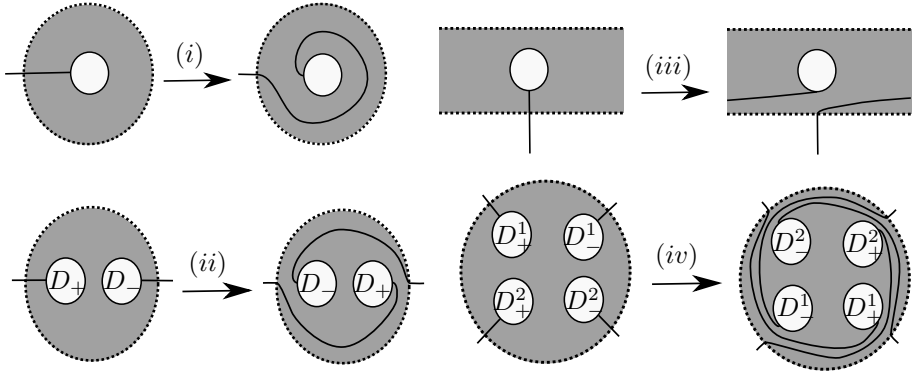
\caption{The four different elementary handle maps from Example~\ref{ex:disk_annulus_twists}. Here we are showing the action of each mapping class on  arcs, and the dotted circles represent the boundaries of the disks or annuli
along which twisting is performed.}
	\label{fig:handle_maps}
\end{figure}

Before we prove the main result, we need a generalization of a classical fact about handlebody groups, cf.~\cite{Oe02}. We consider the homomorphism
\begin{equation}\label{eq:MCG}
   \rho\colon \MCG(K,\partial_- K) \to \MCG(\partial_+ K)\,,
\end{equation}
induced by the restriction. We want to understand the kernel of this homomorphism. To give explicit generators, we define the following diffeomorphisms. Let $c$ be an arc properly embedded in $K$, both of whose endpoints lie in one $2$-sphere component of $\partial_- K$, denoted $S$. Let $T$ denote the boundary of a regular neighborhood of $c \cup S$. Then $T$ is $2$-torus and furthermore there exists an embedded closed curve $c' \subset T$ isotopic to $c$ relative to a regular neighborhood of $S$. The twist along $T$ into the direction of $c'$ is called \emph{slide homeomorphism}. A change of the choice of $c$ in its homotopy class in $\pi_1(K,S)$ changes the slide homeomorphism by an isotopy and possibly a twist along $S$ (cf.~\cite[$\mathsection 2.1$]{MM86}). 

\begin{thm}\label{thm:kernel}
The kernel of $\rho$ in Equation~\ref{eq:MCG} is generated by slide homeomorphisms and twists along $2$-sphere and $2$-torus components of $\partial_- K$. In particular, if there are no $2$-sphere and $2$-torus components, then $\rho$ is injective. 
\end{thm}
\begin{proof}
First, we prove the theorem under the additional assumption that $\partial_- K$ has no $2$-sphere components and explain the general case at the end of the proof. 

We claim that $K$ is aspherical. Since the claim is obvious if $K$ is a handlebody we assume that $\partial_- K \neq \emptyset$. Further, assume by contradiction that there is a non-trivial element in $\pi_2(K)$. The sphere theorem gives us a non-trivial embedded sphere $S$ in $K$. Let $D_i$ be the co-cores of the $1$-handles of $K$. $S$ must intersect the $D_i$ otherwise it represents an element in $\partial_- K \times [0,1]$, which is aspherical by assumption. Now we perform an innermost circle argument to cut $S$ into two $2$-spheres $S_1$, $S_2$ with fewer intersections with the co-cores and such that $[S_1]+[S_2]=[S]$ in $\pi_2(K)$. In particular, at least one of the $[S_i]$ is again non-trivial. Then we proceed by induction on the number of components of the intersections to obtain a contradiction. This shows that $K$ is aspherical.

Now, let $\Phi\colon K\rightarrow K$ be a diffeomorphism that is the identity on $\partial K$. The goal is to show that $\Phi$ is isotopic to a composition of $2$-torus twists along components of $\partial_- K$. For that, we consider a co-core $D$ of a $1$-handle of $K$ and its image $\Phi(D)$. Again, we can use a standard innermost circle argument and the fact that $K$ is aspherical to show that $\Phi(D)$ is isotopic to $D$ relative to the boundary.
 
Thus, we can assume, up to isotopy, that $\Phi$ is the identity on the co-cores of the $1$-handles of $K$. By cutting along the co-cores we get a diffeomorphism $\Phi'$ of $\partial_- K \times [0,1]$ that is the identity on both boundary components. By a theorem of Waldhausen~\cite{Wa68} after a further isotopy, we can assume that $\Phi'$ is level-preserving. Thus $\Phi'$ induces an element in $\pi_1(\mathrm{Diff}(\partial_- K))$. This group is abelian and freely generated by the twists along the $2$-torus boundary components~\cite{EE67}. 
 
Finally, assume that $\partial_- K$ contains $2$-sphere components, denoted $B_1,\dots,B_\ell$. Given a diffeomorphism $\Phi$ which is the identity on $\partial K$. Choose pairwise disjoint arcs $c_1,\dots,c_\ell:[0,1]\to K$, such that $c_j(0) \in B_j$ and $c_j(1) \in \partial_+ K$. The group of slide homeomorphisms acts transitively on the isotopy classes of these arcs with fixed endpoints. Hence we find a composition of slide homeomorphisms $\Phi'$ such that $\Phi'(c_j)$ is isotopic to $\Phi(c_j)$ for all $j$. Note that by using the so-called light bulb trick, we can assume that the traces of the isotopies are disjoint for $i\neq j$. 
Composing $\Phi$ with the inverse of $\Phi'$ and a further isotopy, we can assume $\Phi$ is the identity on $c_j$. Furthermore, by additionally composing with $2$-sphere twists along $B_j$ we can assume that $\Phi$ is the identity restricted to a tubular neighborhood of $c_j$ and a regular neighborhood of $B_j$ for all $j$. Denote by $K'$ the closure of the complement of the union of these neighborhoods. This complement $K'$ is a compression body without $2$-sphere components in its lower genus boundary and $\Phi$ restricted to $K'$ is the identity on $\partial K'$. We conclude as in the first part of the proof. 
\end{proof}

\begin{thm}\label{thm:handle_slides_generate}
Given a compression body $K$ and let $\Phi\colon K \to K$ be a diffeomorphism which is the identity on $\partial_- K$, then $\Phi$ is isotopic relative to $\partial_- K$ to a composition of elementary handle maps.
\end{thm}

\begin{proof}
For the proof, we assume that $\partial_- K \neq \emptyset$. The case $\partial_- K = \emptyset$ is already proven in \cite{Su77}. Let $a_1,\ldots,a_k$ denote the $1$-handles of $K$; $D_i$ and $D_i^\pm$ the cocore disk and respectively the attaching regions of the $1$-handle $a_i$ for $i=1,\ldots,k$; and $B_0,\dots,B_m$ the connected components of $\partial_- K$. Since any two compression bodies with the same topology on the boundary are diffeomorphic and elementary handle maps conjugated by such a diffeomorphism yield elementary handle maps, we assume without loss of generality that for $i\geq 1$ the only attaching region contained in $B_i \times 1$ is $D_i^-$, see Figure~\ref{fig:model_body}. Let $\Phi$ be a diffeomorphism of $K$ which is the identity when restricted to $\partial_- K$. In the following all isotopies of $\Phi$ are understood to be relative to $\partial_- K$.
\begin{figure}
\centering
\def\svgwidth{\columnwidth}
\begingroup%
  \makeatletter%
  \providecommand\color[2][]{%
    \errmessage{(Inkscape) Color is used for the text in Inkscape, but the package 'color.sty' is not loaded}%
    \renewcommand\color[2][]{}%
  }%
  \providecommand\transparent[1]{%
    \errmessage{(Inkscape) Transparency is used (non-zero) for the text in Inkscape, but the package 'transparent.sty' is not loaded}%
    \renewcommand\transparent[1]{}%
  }%
  \providecommand\rotatebox[2]{#2}%
  \newcommand*\fsize{\dimexpr\f@size pt\relax}%
  \newcommand*\lineheight[1]{\fontsize{\fsize}{#1\fsize}\selectfont}%
  \ifx\svgwidth\undefined%
    \setlength{\unitlength}{483.53688848bp}%
    \ifx\svgscale\undefined%
      \relax%
    \else%
      \setlength{\unitlength}{\unitlength * \real{\svgscale}}%
    \fi%
  \else%
    \setlength{\unitlength}{\svgwidth}%
  \fi%
  \global\let\svgwidth\undefined%
  \global\let\svgscale\undefined%
  \makeatother%
  \begin{picture}(1,0.39707849)%
    \lineheight{1}%
    \setlength\tabcolsep{0pt}%
    \put(0,0){\includegraphics[width=\unitlength,page=1]{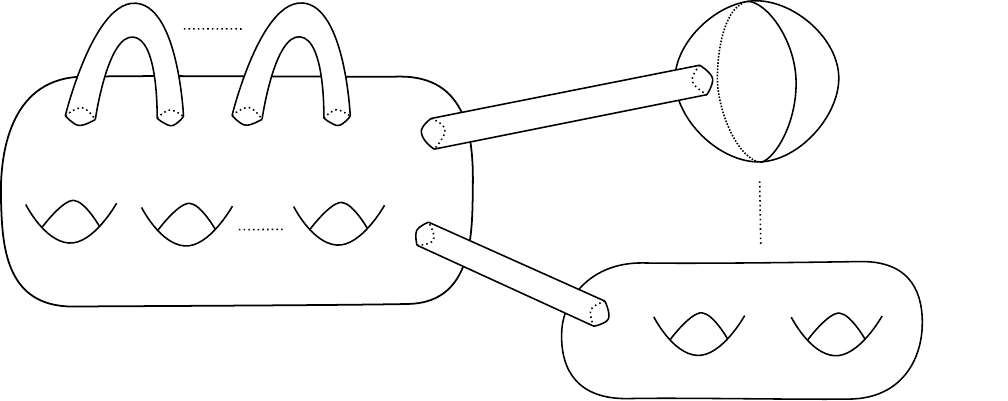}}%
    \put(0.24135092,0.11833303){\color[rgb]{0,0,0}\makebox(0,0)[lt]{\lineheight{0.09}\smash{\begin{tabular}[t]{l}$B_0$\end{tabular}}}}%
    \put(0.84786442,0.31835871){\color[rgb]{0,0,0}\makebox(0,0)[lt]{\lineheight{0.09}\smash{\begin{tabular}[t]{l}$B_1$\end{tabular}}}}%
    \put(0.55796467,0.2503318){\color[rgb]{0,0,0}\makebox(0,0)[lt]{\lineheight{0.09}\smash{\begin{tabular}[t]{l}$a_1$\end{tabular}}}}%
    \put(0.92176786,0.05281061){\color[rgb]{0,0,0}\makebox(0,0)[lt]{\lineheight{0.09}\smash{\begin{tabular}[t]{l}$B_m$\end{tabular}}}}%
    \put(0.47814484,0.08353425){\color[rgb]{0,0,0}\makebox(0,0)[lt]{\lineheight{0.09}\smash{\begin{tabular}[t]{l}$a_m$\end{tabular}}}}%
    \put(0.01987345,0.35431382){\color[rgb]{0,0,0}\makebox(0,0)[lt]{\lineheight{0.09}\smash{\begin{tabular}[t]{l}$a_{m+1}$\end{tabular}}}}%
    \put(0.34894429,0.35254466){\color[rgb]{0,0,0}\makebox(0,0)[lt]{\lineheight{0.09}\smash{\begin{tabular}[t]{l}$a_k$\end{tabular}}}}%
  \end{picture}%
\endgroup%

\caption{The standard model of a compression body $K$ used in the proof of Theorem~\ref{thm:handle_slides_generate}. }
	\label{fig:model_body}
\end{figure}
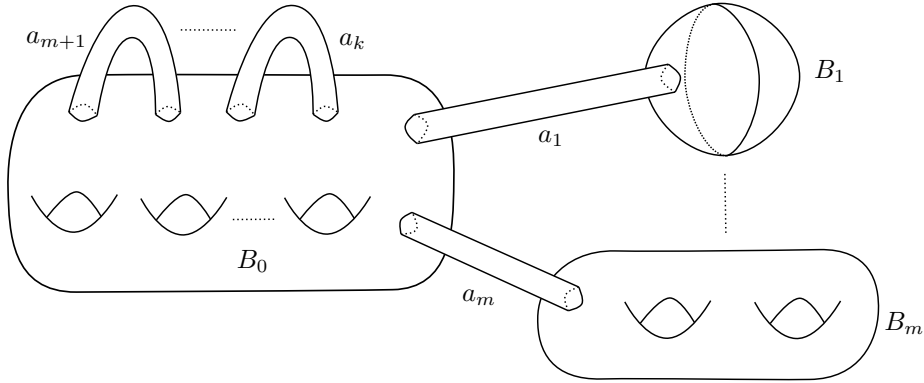

In a first step, we show that up to composing with elementary handle maps the diffeomorphism $\Phi$ is isotopic to the identity when restricted to the belt spheres $c_i:=\partial D_i$ for all $i=1,\dots,k$. 
We proceed by induction on $i$ and implicitly require that all modifications to $\Phi$ are done such that the modified map restricted to $c_j$ is still the identity for all $j<i$. After isotopy of $\Phi$ we assume that the intersection is transverse 
\begin{equation}\label{eq:intersection}
	\Phi(D_i) \cap \big(D_1 \cup D_2 \cup \ldots \cup D_k \big)\,.
\end{equation}
Let $J$ be the union of arcs in~\eqref{eq:intersection}. The closure of each connected component in $\Phi(D_i)\setminus J$ is a disk. Let $\Delta$ be such a disk. By construction the boundary $\partial \Delta$ is contained in $B_\ell \times 1$ for some $\ell =0,\dots,m$, where we have identified $K$ with the quotient $\bigcup_\ell B_\ell \times [0,1]/\sim$ identifying attaching regions $D^\pm_i$. Since $\partial \Delta$ is contractible in $K$ we find a disk $S \subset B_\ell \times 1$ with $\partial S = \partial \Delta$. We call $S$ a \emph{shadow} associated to $\Delta$. In fact, unless $B_\ell$ is a $2$-sphere, the disk $S$ is unique. We choose a shadow for all connected components in $\Phi(D_i) \setminus J$ and denote by $\mathcal{S}$ the set of these shadows. The set $\mathcal{S}$ has the structure of a graph tree, where two shadows are connected by an edge if their corresponding disks share a common arc in their boundaries. A \emph{leaf} is a shadow, which has only one neighboring shadow with respect to the tree structure. 

Now, we inductively remove all but one shadow from $\mathcal{S}$, by modifying $\Phi$ using handle slides and isotopies. If $\mathcal{S}$ contains more than one element, then $\mathcal{S}$ must contain leafs. Let $S$ be such a leaf. Then $\partial S$ intersects only one attaching region, say $D_j^\epsilon$. We separate the argument depending on two different cases. 
\begin{figure} 
\centering
\def\svgwidth{\columnwidth}
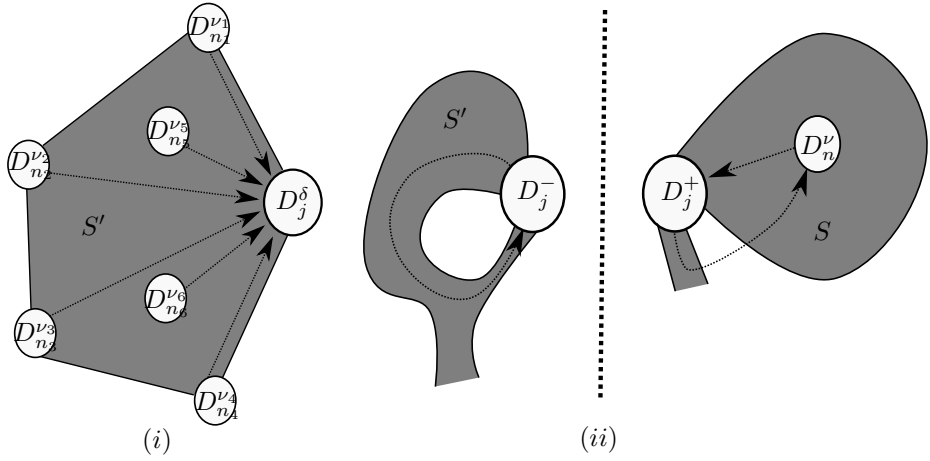
\caption{Elementary handle maps that simplify the complexity in the proof of Theorem~\ref{thm:handle_slides_generate}. }
	\label{fig:proof_handle_maps}
\end{figure}
\begin{enumerate}[(i)]
    \item  $D_j^\epsilon$ is non-separating. We claim that by passing to neighboring shadows, we find a shadow $S'$ contained in $S$, which is disjoint from $D_j^{-\delta}$ and intersects $D_j^\delta$ along its boundary for some $\delta=\pm$. Indeed, if that is not already the case for $S$ we must have $D_j^{-\epsilon} \subset S$. Let $S_1$ be the neighboring shadow of $S$. We have $\partial S_1 \cap D_j^{-\epsilon} \neq \emptyset$ and $S_1 \subset S$. If the claim is not true for $S_1$ we must have $\partial S_1 \cap D_j^\epsilon \neq \emptyset$ since we can not have $D_j^\epsilon \subset S_1$ as $S_1 \subset S$. Thus we find $S_2$ connected to $S_1$, satisfying $\partial S_2 \cap D_j^{-\epsilon} \neq \emptyset$ and $S_2 \subset S_1$. We inductively define new shadows $S_3, S_4, \ldots$ assuming each step, that we do not obtain a shadow with the claimed properties. This process eventually has to stop, since $S_{\ell+1} \subset S_\ell$ for all $\ell$. The shadow $S'$ is possibly no longer a leaf, but still we find arcs for all $D_n^\nu$ with $D_n^\nu \cap S' \neq \emptyset$, contained in $\partial_+ K$ starting and ending on $D_n^\nu$, traversing $a_j$ once, disjoint from other handles and intersecting $\Phi(D_i)$ exactly once, see $(i)$ of Figure~\ref{fig:proof_handle_maps}. Handle slides along these paths and an isotopy removes $S'$ from $\mathcal{S}$ without creating new arcs in~\eqref{eq:intersection}.
    \item $D_j^\epsilon$ is separating. If $S$ does not contain attaching regions, we remove $S$ using an isotopy. Thus we assume that $S$ contains attaching regions and thus $S \subset B_0 \times 1$, by our assumptions on $K$. Let $S'$ be the neighboring shadow of $S$, which can not contain attaching regions in its interior, since $S' \subset B_j \times 1$ with $j\geq 1$. Then either $S'$ is also a leaf and can be removed using an isotopy, or for each $D_n^\nu \subset S$ there exists an arc in $\partial_+ K$ with endpoints in $D_n^\nu$, which traverses $a_j$ twice with different orientations, disjoint from other handles, intersecting $\partial S$ transversely once and $\Phi(c_i)$ only after it came back, see Figure~\ref{fig:proof_handle_maps}. Handle slides along these paths and an isotopy removes $S$ from $\mathcal{S}$ without creating new arcs in~\eqref{eq:intersection}. 
\end{enumerate}
Repeating the steps often enough we assume that $\mathcal{S}=\{S\}$. We want to show $S$ contains exactly one attaching region given by $D_i^\pm$. Again we separate the argument.
\begin{enumerate}[(i)]
    \item $D_i$ is non-separating. Then $\Phi(D_i)$ is also non-separating. If $S$ contains none, only attaching regions of separating handles or both attaching regions of non-separating handles then $\Phi(D_i)$ is separating, in contradiction to the previous observation. Hence it must contain a single attaching region of a non-separating handle, say $D_j^+$. For any other $D_n^\nu$ contained in $S$ we find an arc contained in $\partial_+ K$ with boundary on $D_n^\nu$, traversing $a_j$ once, disjoint from other handles and intersecting $\Phi(c_i)$ transversely once. After concatenating $\Phi$ with the handle slides along the arcs, we assume that $S$ contains only $D_j^+$. We conclude that after isotopy $\Phi(c_i) = c_j$. After exchanging the handle we assume that $\Phi(c_i) = c_i$. Finally, we possibly apply inversion of the handle to show that $\Phi$ is the identity when restricted to $c_i$. 
    \item $D_i$ is separating. Because $\Phi$ is the identity on $\partial_- K$, the space $\Phi(B_i \times [0,1])$ deformation retracts to $B_i$. Attaching the handles contained in $S$ to $S \times [0,1]$ gives $\Phi(B_i \times [0,1])$. If $S$ does not contain $D^\pm_i$, then $\Phi(B_i \times [0,1])$ or its complement contains $B_i \cup B_0$, in contradiction to the previous observation. If $S$ contains another attaching region, then $\Phi(B_i \times [0,1])$ does not contract to $B_i$. This shows that up to isotopy $\Phi(c_i) = c_i$ and thus that $\Phi$ is the identity when restricted to $c_i$, since $\Phi$ can not act as inversion on $c_i$. 
\end{enumerate}
Hence, by induction we conclude that $\Phi$ is the identity on $c_i$ for all $i=1,\dots,k$. 

We show now that up to composing with elementary handle maps, $\Phi$ is the identity when restricted to $\partial_+ K$. Consider the closure of the surface $B_\ell \times 1 \setminus \bigcup_j D^\pm_j$, denoted $B_\ell'$. Because $\Phi$ is the identity on the boundary, the map $\Phi$ preserves $B_\ell'$ setwise. Let $c \subset B_\ell$ be a non-contractible embedded curve. As $\Phi$ is the identity on $\partial_- K$ the concatenated annuli $c \times [0,1] \cup_{c\times 0} \Phi(c\times [0,1])$ show that $c$ is isotopic to $\Phi(c)$. Thus $\Phi$ restricted to $B_\ell \times 1$ is isotopic to the identity. The kernel of the homomorphism on mapping class groups, induced by the inclusion $B_\ell' \to B_\ell \times 1$ is generated by twisting and sliding the attaching regions. Hence after composing these with $\Phi$, we assume that $\Phi$ is also isotopic to the identity when restricted to $B_\ell'$. Concluding we have shown that up to applying elementary handle maps, the diffeomorphism $\Phi$ restricted to $\partial_+ K$ is isotopic to the identity. With Theorem~\ref{thm:kernel} it remains to show that slide-homomorphisms and twists along $2$-sphere or $2$-torus components of $\partial_- K$ can be expresses as elementary handle maps. 

Let $\Phi_T$ be a torus twist in direction of $c \subset T$ with $T=B_j \times \frac 12$ for some $2$-torus $B_j \subset \partial_-K$. Fix an identification $T\cong S^1 \times S^1$ and strictly increasing numbers $s_1,\dots,s_n \in (0,2\pi)$ such that $c \cong 0 \times S^1$ and the annuli $A_i = [s_i,s_{i+1}] \times S^1 \subset T$  with $s_{n+1}:=s_1 +2\pi$ contain each at most one attaching region. We isotope  $A_i$ along the projection $B_j \times [0,1] \to B_j \times 1$ to proper embedded annuli with boundary on $\partial_+ K$, still denoted $A_i$, and conclude that $\Phi_T = \Phi_{A_1}\circ \dots \circ \Phi_{A_n}$. This shows that $\Phi_T$ is a product of elementary handle maps. Similarly we proceed to show that twists along $2$-sphere components of $\partial_- K$ are isotopic to a product of elementary handle maps. Finally we have to show the same for slide homeomorphisms. Let $\Phi$ be a slide homeomorphism of a $2$-sphere component $S \subset \partial_- K$ defined with respect to an arc $c$ with boundary on $S$. Since composition of slide homeomorphisms of $S$ correspond to concatenation of arcs it suffices to consider an arc $c$ which corresponds to a generator under the identification $\pi_1(K,S) \cong \pi_1(K,p)$ for some $p \in S$. There are two cases:
\begin{enumerate}[(i)]
    \item $c$ intersects non-separating $D_i$ transversely in exactly one point and is disjoint of all other handles. By simplifying assumptions $S=B_0$ and after isotopy $k_i:=c\cap a_i$ is the core of $a_i$ and $c \cap (B_0 \times [0,1])= \partial k_i \times [0,1]$. Let $T = \partial \nu(S \cup c)$. Cutting the torus $T$ as above, we conclude that $\Phi$ is given as the product of sliding all feet of all other handles attached to $B$ over $a_i$. 
    \item $c$ intersects separating $D_i$ transversely in exactly two points and intersection points with $D_j$ for $j\neq i$ are in between these two points with respect to the order on $c$. We assume without loss of generality that $S=B_j$ with $j\geq 1$ and after isotopy $c \cap B_j \times [0,1] = \partial c \times [0,1]$. We find an arc $c' \subset \partial_+(K \setminus a_j)$ with boundary on $D_j^+$ isotopic to $c \setminus (B_j \times [0,1])$ relative to $D_j^+$.  Let $T=\partial \nu (S \cup c)$. Cutting the torus we see that $\Phi$ is given by sliding $D_j^+$ along $c'$. 
\end{enumerate}
This concludes the proof. 
\end{proof}

\begin{cor}\label{cor:twists_generate}
Every element in $\MCG(K,\partial_- K )$ is represented as a product of twists along disks and annuli and half-twists along disks.
\end{cor}

\begin{proof}
This follows from Theorem~\ref{thm:handle_slides_generate} since every elementary handle map is given via twists along disks or annuli and half-twists. 
\end{proof}

\subsection{Heegaard splittings} Let $M$ be a compact $3$-manifold possibly with non-empty boundary $B=\partial M$. A \emph{Heegaard splitting} of $M$ is a pair of compression bodies $(K,H)$ such that $K \cup H = M$, $K \cap H = \partial_+ K = \partial_+ H $, $\partial_- K = B$ and $\partial_- H= \emptyset$. More generally if $M$ is a sutured manifold with suture $\partial M = B \cup_C B'$ then a \emph{sutured Heegaard splitting} is a pair of compression bodies $(K,H)$ with $H \cup K = M$, $H \cap K = \partial_+ H = \partial_+ K$, $\partial_- K = B$ and $\partial_- H = B'$. The surface $F=K \cap H$ is called \emph{Heegaard surface} of the splitting. Note that for a sutured Heegaard splitting we have $\partial B \cong \partial F \cong \partial B'$. The \emph{Heegaard genus} of the given Heegaard decomposition is the genus of its Heegaard surface. 

For a Heegaard splitting there is a stabilization procedure as follows: Given a Heegaard splitting $(K,H)$ of $M$ choose a properly embedded arc $\gamma$ in $K$ with endpoints on $F$ which is boundary parallel. Let $N$ denote the closed $3$-dimensional regular neighborhood of $\gamma$. We obtain a new Heegaard splitting by putting
\[ H' = H \cup N,\qquad K' = K \setminus \mathring{N}\,.\]
The operation of replacing $(K,H)$ with $(K',H')$ is called \emph{stabilization} of the Heegaard splitting. We say that a diffeomorphism $\Phi\colon M \to M$ \emph{preserves} the Heegaard splitting $(K,H)$ if $\Phi(K) = K$ (and thus $\Phi(H)=H$). We show in Proposition~\ref{prop:splitting} below, that it is actually always possible to find a Heegaard splitting that is preserved by a given diffeomorphism. 

A way to obtain a Heegaard splitting is to consider a handle decomposition relative to the boundary
\begin{equation}\label{eq:handle_decomposition}
 M = B \times [0,1] \cup a_1 \cup \dots \cup a_k \cup b_1 \cup \dots b_\ell \cup c\,,
\end{equation}
consisting of $k$ many $1$-handles $a_i$, $\ell$ many $2$-handles $b_j$ and a single $3$-handle $c$. A Heegaard splitting of $M$ is given by setting
\[ K = B \times I \cup a_1 \cup \ldots \cup a_k\,\qquad H = b_1 \cup \ldots \cup b_\ell \cup c\,.\]
Conversely, given a Heegaard splitting $(K,H)$ we can always find a handle decomposition as in Equation~\eqref{eq:handle_decomposition} that induces the Heegaard splitting.  

A \emph{Heegaard diagram} consists of the data $(F,\delta, \epsilon)$, where $F$ is the Heegaard surface, $\delta=(\delta_j)$ and $\epsilon=(\epsilon_i)$ denote the attaching spheres for the $2$-handles of $H$ and $K$, respectively. 

A Heegaard diagram completely determines a Heegaard splitting up to diffeomorphism. It is known that any two Heegaard splittings of the same $3$-manifold admit a common stabilization~\cite{Re33,Si33}. The following well-known proposition shows that $g_\Phi(M)$ which was used in the statement of Theorem~\ref{thm:A} is well-defined.

\begin{prop}\label{prop:splitting}
	For every diffeomorphism $\Phi\colon M \to M$ with $\Phi|_{\partial M}= \mathrm{id}$ there exists a Heegaard splitting of $M$ that is preserved by $\Phi$.
\end{prop}

\begin{proof} We start with some handle decomposition $\mathcal{H}$ of $M$ relative to its boundary with a single $3$-handle. We denote by $\mathcal{H}'$ the handle decomposition of $M$ which is given by the image of $\mathcal{H}$ under $\Phi$. By Cerf theory~\cite[Section~4]{GK15} we know that $\mathcal{H}$ and $\mathcal{H}'$ are related by finitely many handle slides and removing/introducing canceling $(1,2)$-handle pairs. It follows that the two Heegaard splittings induced by the handle decompositions $\mathcal H$ and $\mathcal H'$ admit a common stabilization that is isotopic to a Heegaard splitting which is preserved by $\Phi$.
\end{proof}

\begin{rem}
On the other hand, the genus of the constructed Heegaard splitting from Proposition~\ref{prop:splitting} will in general not be minimal. In fact, one can show that there exist examples where $g(M)<g_\Phi(M)$, see for example~\cite[Section~3]{Ta20}.
\end{rem}

\subsection{Trisections} Given integers $0 \leq k \leq g$. A \emph{$(g,k)$-trisection} of a closed $4$-manifold $W$ is a decomposition 
\begin{equation}\label{eq:trisection_decomposition}
W=W_0\cup W_1 \cup W_2\,,
\end{equation}
such that 
\begin{enumerate}[(i)]
    \item each $W_i$ is a $4$-dimensional $1$-handlebody $\natural_k (S^1 \times B^3)$
    \item for $i \neq j$, $W_i \cap W_j$ is a $3$-dimensional $1$-handlebody $\natural_g (S^1 \times D^2)$; and
    \item the triple intersection $\Sigma = W_0 \cap W_1 \cap W_2$ is a closed surface of genus $g$. 
\end{enumerate}
More generally, given non-negative integers $g,k,p,b$ and a compact $4$-manifold $W$ with non-empty boundary $\partial W$.  A \emph{relative $(g,k;p,b)$-trisection} of $W$ is a decomposition~\eqref{eq:trisection_decomposition} such that 
\begin{enumerate}[(i)]
    \item each $W_i$ is diffeomorphic to $[0,1]\times K$ where $K$ is a compression body with $\partial_+ K$ a surface of genus $k$ 
     and $\partial_- K=P$ a surface of genus $p$ both with $b$ boundary components, 
    \item for $i\neq j$, $W_i \cap W_j$ is a compression body with lower and higher genus boundary given by $P$  and $\Sigma$ respectively, where
    \item $\Sigma=W_0 \cap W_1 \cap W_2$ is a genus $g$ surface with $b$ boundary components,
    \item $W_i \cap \partial W$ diffeomorphic to $P \times I$; and
\item $(W_{i-1} \cap W_i) \cup  (W_i \cap W_{i+1})$ is a sutured Heegaard splitting of $\mathrm{cl}(\partial W_i \setminus \partial W)$, where the suture is given by the image of the two copies of $\partial P$ in $(W_{i-1} \cap W_i)$ and $(W_i \cap W_{i+1})$ under the previous diffeomorphisms.
\end{enumerate}
Note that a relative trisection induces on $\partial W$ either an $S^1$-fibration with fiber $P$ if $b=0$ or an open book decomposition with page $P$ if $b>0$. The surface $\Sigma$ is called \emph{trisection surface} and the \emph{genus} of the trisection is $g=g(\Sigma)$. Further note that $\Sigma$ is at the same time a Heegaard surface for the Heegaard splitting of $\mathrm{cl}(\partial W_i \setminus \partial W)$ given by
\begin{equation}\label{eq:trisection_Heegaard_splitting}
(W_{i-1} \cap W_i) \cup  (W_i \cap W_{i+1}) = (\partial W_{i-1} \cap \partial W_i) \cup (\partial W_i \cap \partial W_{i+1})\,.
\end{equation}
There is also a stabilization procedure for a (relative) trisection as follows. Given a trisection $(W_0,W_1,W_2)$ of $W$ possibly with non-empty boundary. For $i,j \in \{0,1,2\}$ with $i \neq j$ choose properly embedded arcs $\gamma_{ij}$ in $W_i \cap W_j$ which are pairwise disjoint, boundary parallel and have endpoints on $\Sigma$. Let $N_{ij}$ be the closed $4$-dimensional regular neighborhood of $\gamma_{ij}$ in $W$, which are also assumed to be pairwise disjoint. We define a new trisection through
\begin{itemize}
    \item $W_0' = (W_0 \cup N_{12}) \setminus \big(\mathring{N}_{01} \cup \mathring{N}_{02} \big)$
    \item $W_1' = (W_1 \cup N_{02}) \setminus \big(\mathring{N}_{01} \cup \mathring{N}_{12} \big)$
    \item $W_2' = (W_2 \cup N_{01}) \setminus \big(\mathring{N}_{02} \cup \mathring{N}_{12} \big)$
\end{itemize}
The operation of replacing $(W_0,W_1,W_2)$ with $(W'_0,W'_1,W'_2)$ is called \emph{stabilization}. Stabilizations of this kind do not change the open book decompostion induced on the boundary. There is also a \emph{relative stabilization} procedure, as explained in \cite{CIMT22}, which does change the open book decomposition induced on the boundary. It is known that any $4$-manifold admits a trisection and that any two trisections on the same manifold are equivalent up to stabilization and relative stabilization~\cite[Theorem~11 and~21]{GK16}, \cite{CIMT22}.  

A \emph{trisection diagram} is a quadruple $(\Sigma,\alpha,\beta,\gamma)$ where each triple $(\Sigma,\alpha,\beta)$, etc.\ is a Heegaard diagram corresponding to the (sutured) Heegaard splitting~\eqref{eq:trisection_Heegaard_splitting} for $i=0,1,2$ respectively. A trisection diagram uniquely determines the diffeomorphism type of the underlying $4$-manifold~\cite{GK16}.  

\section{Proof of Theorem~\ref{thm:A}}\label{sec:proof}
We assume first that  $W$ has the structure of an open book, possibly with boundary. The proof when $W$ is a $S^1$ bundle is very similar. We explain at the end what needs to be changed. The Heegaard splitting $M=K\cup_F H$ induces a decomposition of $W$ into an open book with boundary $W_K$ and a mapping cylinder $W_H$
\begin{equation}\label{eq:splitting_of_W}
   W_K := \mathrm{Ob}(K,B\cup_C F,\Phi) \,,\qquad W_H := H_\Phi \,. 
\end{equation}
Note that by abuse of notation, we use $\Phi$ to also denote its restrictions to $K$ and $H$. Furthermore, the spaces $W_K$ and $W_H$ are canonically subspaces of $W$, and their intersection is given by the mapping cylinder $F_\Phi$. The trisection on $W$ is constructed in the following steps.
\begin{enumerate}[(1)]
	\item We construct a surface $\Sigma$ in $W_K$ with $g(\Sigma)=\ell + 2k$, where $k$ and $\ell$ are the number of $1$-handles of $K$ and $H$ respectively. This surface will be the trisection surface.
	\item We construct compression bodies $K_0$, $K_1$ and $K_2$ in $W_K$ such that $\partial_+ K_j = \Sigma$ and $\partial_- K_j = F_{2j\pi/3}$ where $F_\theta$ denotes the fibre of the mapping cylinder $F_\Phi$. 
\item We prove that we obtain a relative $(2k+\ell,0;\ell,b)$-trisection of $W_K$, where $b$ denotes the number of boundary components of $F$, by showing that a regular neighborhood of $\bigcup_j K_j$ is isotopic to (a slightly smaller) $W_K$.
\item We attach three copies of $H$, denoted $H_j$, to $K_j$ for $j=0,1,2$. By the last step, each connected component of the complement is a trivial $H$-bundle over the interval, hence we obtain a $(2k+\ell,\ell)$-trisection (resp.\ a relative $(2k+\ell,\ell,p,b)$-trisection in the case when $W$ has boundary) as required, with double intersection given by $K_j \cup_{F_j} H_j$ with $F_j = F_{2j\pi/3}$ and triple intersection $\Sigma$. 
\end{enumerate}
Finally, we make this process concrete enough to obtain diagrams for the trisections. 

\subsection{Trisection surface} We fix a relative handle decomposition
\begin{equation}\label{eq:K_handle_decomposition}
 K = B\times I \cup a_1 \cup \dots \cup a_k \,
\end{equation}
from which we obtain an open book decomposition of $W_K$
\begin{equation}\label{eq:W_decomposition}
 W_K = \left. B \times D^2 \cup \left(B \times I \cup a_1 \cup \dots \cup a_k\right) \times [0,2\pi] \right/ \sim\,,
\end{equation}
with equivalence relation $(p,2\pi) \sim (\Phi(p),0)$. For $\theta \in [0,2\pi]$ let $K_\theta \subset W_K$ be the fibre of mapping torus of $\Phi$, i.e.\ the $\theta$-page. For any $i=1,\ldots,k$, $z \in D^2$ and $\theta \in [0,2\pi]$ we are going to define arcs $c^\theta_{i,z}$ in $W_K$ with endpoints on $B$, which outside of $B\times D^2$ lie inside $K_\theta$ and run parallel to the core of the handle corresponding to $a_i$. To that end fix an identification of the $1$-handles $a_i$ with $[-1,1] \times D^2$ and define points $p^\pm_i\in B$ such that $(p^\pm_i,1)$ is identified with $(\pm 1,z)$, the subset $I_\theta \subset D^2$ to be the line segment from $0$ to $e^{\sqrt{-1}\theta}$ and finally using decomposition~\eqref{eq:W_decomposition} above (cf.\ Figure~\ref{figure:extended_core}) 
\begin{figure} 
\centering
\def\svgwidth{0.7\columnwidth}
\begingroup%
  \makeatletter%
  \providecommand\color[2][]{%
    \errmessage{(Inkscape) Color is used for the text in Inkscape, but the package 'color.sty' is not loaded}%
    \renewcommand\color[2][]{}%
  }%
  \providecommand\transparent[1]{%
    \errmessage{(Inkscape) Transparency is used (non-zero) for the text in Inkscape, but the package 'transparent.sty' is not loaded}%
    \renewcommand\transparent[1]{}%
  }%
  \providecommand\rotatebox[2]{#2}%
  \newcommand*\fsize{\dimexpr\f@size pt\relax}%
  \newcommand*\lineheight[1]{\fontsize{\fsize}{#1\fsize}\selectfont}%
  \ifx\svgwidth\undefined%
    \setlength{\unitlength}{448.28038657bp}%
    \ifx\svgscale\undefined%
      \relax%
    \else%
      \setlength{\unitlength}{\unitlength * \real{\svgscale}}%
    \fi%
  \else%
    \setlength{\unitlength}{\svgwidth}%
  \fi%
  \global\let\svgwidth\undefined%
  \global\let\svgscale\undefined%
  \makeatother%
  \begin{picture}(1,0.66231089)%
    \lineheight{1}%
    \setlength\tabcolsep{0pt}%
    \put(0,0){\includegraphics[width=\unitlength,page=1]{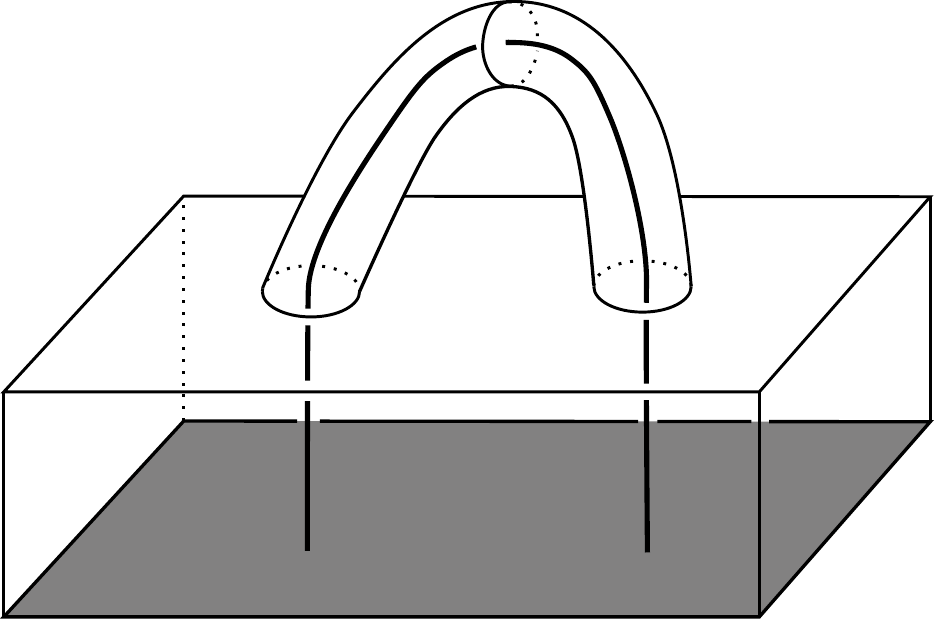}}%
    \put(0.71272281,0.55861893){\color[rgb]{0,0,0}\makebox(0,0)[lt]{\lineheight{0.1}\smash{\begin{tabular}[t]{l}$a_i$\end{tabular}}}}%
    \put(0.53655551,0.58031797){\color[rgb]{0,0,0}\makebox(0,0)[lt]{\lineheight{0.1}\smash{\begin{tabular}[t]{l}$z$\end{tabular}}}}%
    \put(0.16361304,0.05480236){\color[rgb]{0,0,0}\makebox(0,0)[lt]{\lineheight{0.1}\smash{\begin{tabular}[t]{l}$B$\end{tabular}}}}%
    \put(0.70680672,0.2776342){\color[rgb]{0,0,0}\makebox(0,0)[lt]{\lineheight{0.1}\smash{\begin{tabular}[t]{l}$c^{\theta}_{i,z}$\end{tabular}}}}%
    \put(0,0){\includegraphics[width=\unitlength,page=2]{core.pdf}}%
  \end{picture}%
\endgroup%

\caption{The definition of the arc $c_{i,z}^\theta$. Here we draw a $\theta$-page, where we identify $K_\theta$ with $K$.}
	\label{figure:extended_core}
\end{figure}
\begin{equation}\label{eq:extended_core}
   c^\theta_{i,z} := \{p^-_i,p^+_i\} \times  I_\theta \cup \big( \{p^-_i,p^+_i\} \times I \cup [-1,1] \times z\big) \times \theta\,.
\end{equation}
Note that outside of $B\times D ^2$ the arcs $c^{2\pi}_{i,z}$ and $c^0_{i,z}$ are contained in the same fibre $K_0$ and are related by
\begin{equation}\label{eq:monodromie_on_cores}
 c^{2\pi}_{i,z} \cap K_0 = \Phi (c^0_{i,z}\cap K_0)\,,
\end{equation}
where we identified $K_0$ with $K$. The trisection surface is given by surgering $B$ with handles given by small $3$-dimensional tubes with cores $c^{\theta}_{i,z}$ in three different angles. More precisely fix points $z_j:=\frac 12 e^{2\sqrt{-1}j \pi/3}\in D^2$ for all $j\in \Z$ and for any $z \in D^2$ and $\rho >0$ let $D_\rho(z) = \{ w \in D^2 \,:\, |w-z| \leq \rho \}$ be the closed disk of radius $\rho$ around $z$. Then define the trisection surface, setting $\rho=1/4$ 
\[\Sigma := B \setminus \bigcup_{\substack{i=1,\ldots,k\\j=0,1,2}} \limits \bigcup_{z \in D_\rho(z_j)} \partial c^{2j \pi/3}_{i,z}  \cup \bigcup_{\substack{i=1,\ldots,k\\j=0,1,2}} \bigcup_{z \in \partial D_\rho(z_j)} c^{2j \pi/3}_{i,z}\,. \]
\subsection{Compression bodies} We define the compression bodies by specifying two families of compression disks in $W_K$ with boundary on $\Sigma$. The first family is given using decomposition~\eqref{eq:W_decomposition} and the identification $a_i \cong [-1,1] \times D^2$ as above for $i=1,\ldots,k$ and $j=0,1,2$
\begin{equation}\label{eq:Delta1}
\Delta^1_{i,j} = 0 \times D_\rho(z_{j-1}) \times \big[2  (j-1)\pi/3\big]\,,
\end{equation}
where $[\theta] \in [0,2\pi)$ is the unique number such that $\theta - [\theta] \in 2\pi \Z$. In other words, these are the cocores of the handles which were added to $B$ to obtain $\Sigma$. 

We now define the second family. First note that the complement of $\bigcup_j D_\rho(z_j)$ cuts the circle of radius $1/2$ into three arcs, where we denote by $\mu_j$ the arc from $D_\rho(z_{j})$ to $D_\rho(z_{j+1})$ for $j=0,1,2$. Further denote by $w^-_j \in \partial D_\rho(z_j)$ and $w^+_j \in \partial D_\rho(z_{j+1})$ the start point resp.\ end point of $\mu_j$. We define for $j=0,1$ 
\begin{equation}\label{eq:Delta2}
	\Delta^2_{i,j} := \bigcup_{\theta \in [2j\pi/3,  2(j+1)\pi/3]} c^\theta_{i,w_j^-} \cup \bigcup_{z \in \mu_j} c^{2(j+1)\pi/3}_{i,z} \,.  
\end{equation}
To define the disk $\Delta^2_{i,2}$ we need a bit more work incorporating the monodromy. Using again the identification of the handle $a_i$ with $[-1,1]\times D^2$ we define smaller handles inside $a_i$ for some $z \in D^2$ via
\[
a_{i,z}:= [-1,1] \times D_\rho(z)\,.
\]
The following lemma constructs a particular Heegaard splitting of $K$ which will be important for us to define $\Delta^2_{i,2}$
\begin{lem}\label{lem:icebreaker}
For any $z,w \in D^2$ with $|z-w| > 2\rho$ and diffeomorphism $\Phi:K \to K$, there exists a permutation $\sigma$ of $\{1,\dots,k\}$ such that after suitable isotopy of $\Phi$ relative to $\partial_- K$, we have a Heegaard splitting of $K$ given by the handle decomposition
\[B\times [0,1] \cup  a_{1,z} \cup \ldots \cup a_{k,z} \cup \Phi(a_{1,w}) \cup \ldots \cup \Phi(a_{k,w}) \cup b_1 \cup \ldots \cup b_k\,,\]
with $2k$ many $1$-handles $a_{i,z}$ and $\Phi(a_{i,w})$ and $k$ many $2$-handles $b_i$ such that for all $i=1,\ldots,k$ the attaching sphere of $b_i$
\begin{enumerate}[(i)]
    \item intersects transversely the belt sphere of $\Phi(a_{i,w})$ in exactly one point and is disjoint from $\Phi(a_{j,w})$ for all $j\neq i$, and
    \item intersects transversely the belt sphere of $a_{\sigma(i),z}$ in exactly one point.
\end{enumerate}
\end{lem}
\begin{proof}
    By Theorem~\ref{thm:handle_slides_generate} we assume that after a suitable isotopy the diffeomorphism $\Phi$ is given as the composition of elementary handle maps $\Phi_1 \circ \dots \circ \Phi_n$ and that $a_{i,z}$ and $\Phi(a_{j,w})$ are pair-wise disjoint. It suffices to construct the $2$-handles $b_i$ with the required properties. We proceed by induction on $n$. For $n=1$ the handles are constructed in a local model. Note that the attaching sphere $Q_i$ of $b_i$ is uniquely determined by its properties up to isotopy. For the convenience of the reader, we just indicated the cores of the handles in Figure~\ref{fig:awesome_pics}.
    \begin{figure}
    \centering
    \def\svgwidth{0.99\columnwidth}
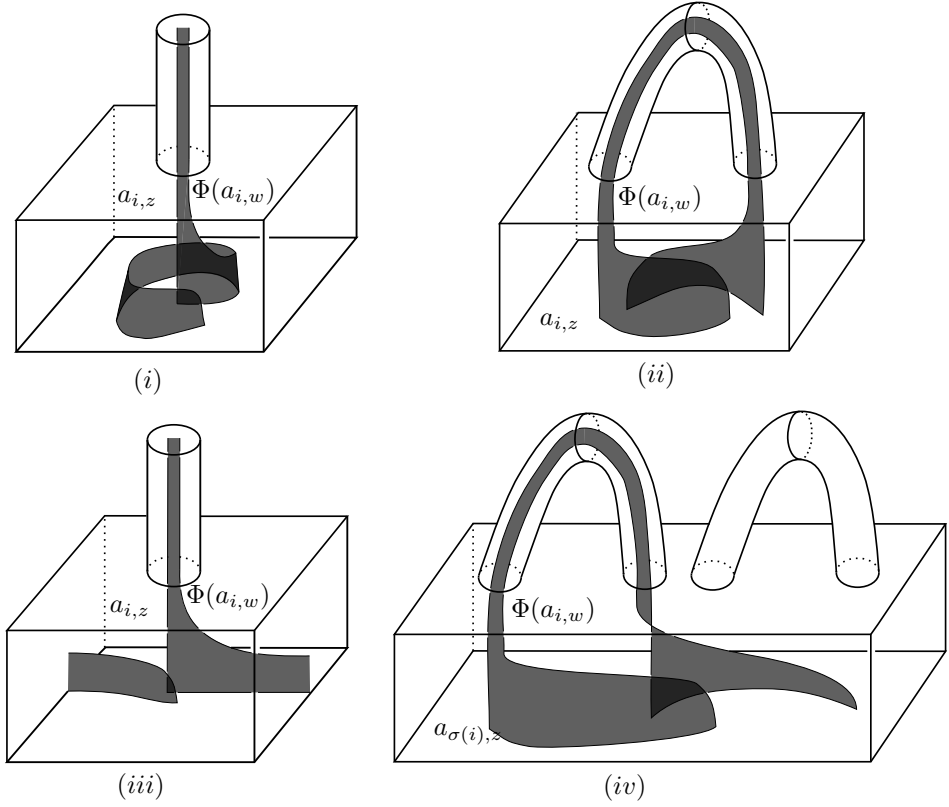\caption{\label{fig:awesome_pics}The cores of the handles $b_i$ for the four elementary handle maps. In picture $(iii)$ the right and left side of the cuboid are identified.}
    \end{figure}
    To conclude for $n \geq 2$, we assume we have already constructed $2$-handles $b'_i$ satisfying the assertion of the lemma with respect to points $z_0,z_1$ and diffeomorphism $\Phi':= \Phi_1\circ \dots \circ \Phi_{n-1}$, where again $z_j = \frac 12 e^{2\sqrt{-1}j \pi/3}$ for $j \in \Z$. Let $b''_i$ denote the $2$-handles satisfying the assertion of the lemma with respect to points $z_1,z_2$ and diffeomorphism $\Phi_n$. Applying $\Phi'$ to the second handle decomposition gives a handle decomposition with $1$-handles $\Phi'(a_{i,z_1})$ and $\Phi(a_{i,z_2})$ and $2$-handles $\Phi'(b''_i)$. Canceling the handles $\Phi(a_{i,z_2})$ and $\Phi(b'_i)$ in the second handle decomposition gives a diffeomorphism $\Psi$ of $K$, which identifies a small neighborhood of $\Phi'(a_{i,z_1})$ with a small neighborhood of $\Phi'(a_{i,z_1}) \cup Q_i'' \cup \Phi(a_{i,z_2})$, where $Q_i''$ is the core of the handle $\Phi'(b''_i)$. Without loss of generality we can assume that $\Psi$ is the identity along $Q_i' \cap \Phi'(a_{i,z_1})$, where $Q_i'$ is the core of $b_i'$. If we pull back the first handle decomposition using $\Psi$ and adding to it the second one, we obtain a handle decomposition of $K$ with $3k$ many $1$-handles $a_{i,z}$, $\Phi'(a_{i,z_1})$ and $\Phi(a_{i,z_2})$ and $2k$-many $2$-handles given by $\Psi^{-1}(\Phi'(b_i''))$ and $b_i'$. Now cancelling $b_i'$ and $\Phi'(a_{i,z_1})$ yields $2$-handles $b_i$ satisfying the required assertions. Note that after handle isotopy we have such a handle decomposition for arbitrary $z,w$ and not just for $z_0$ and $z_2$. 
\end{proof}

\begin{rem}\label{rem:icebreaker}
	The previous Lemma is constructive. The attaching sphere of the handle $b_i$ is obtained inductively from the models in Figure~\ref{fig:awesome_pics}. For the induction step, write in a model of $B$ the attaching regions of the $1$-handles $a_{i,z_0}$, $\Phi'(a_{i,z_1})$ and $\Phi(a_{i,z_2})$ and the attaching spheres of the $2$-handles $b_i'$ and $\Phi'(b_i'')$ using notation as introduced in the proof of the Lemma. Now the attaching spheres of $b_i'$ and $\Phi'(b_i'')$ possibly intersect. We remove these intersection points by changing $b_i'$ using finger moves as indicated in Figure~\ref{fig:finger_move}. The result gives a Heegaard diagram of $K$ and we cancel the handles $b_i'$ and $\Phi'(a_{i,z_1})$ to obtain the required handle decomposition.
\begin{figure}
 \centering
 \def\svgwidth{.6\columnwidth}
 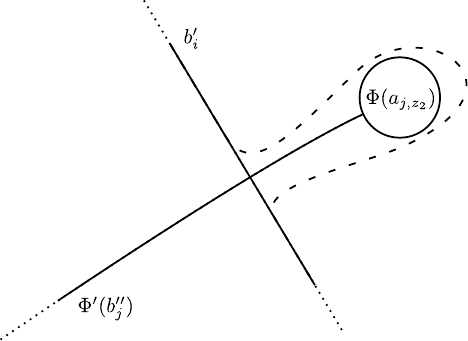
 \caption{We remove intersection points by changing the attaching sphere of $b_i'$ as indicated by the dashed line.}
 	\label{fig:finger_move}
 \end{figure}
\end{rem}
Let $Q_i \subset K$ denote the core of the handle $b_i$ and $\sigma$ the permutation obtained from Lemma~\ref{lem:icebreaker} with respect to the points $w=z_2$ and $z=z_0$ and monodromy $\Phi$. In light of~\ref{eq:monodromie_on_cores} we assume after possibly a further isotopy of $Q_i$ that
\[ \partial Q_i \cap \Phi(a_{i,z_2}) = c_{i,w_2^-}^{2\pi}\cap K_0\,,\qquad \partial Q_i \cap a_{\sigma(i),z_0} = c_{\sigma(i),w_2^+}^{0}\cap K_0\,.\]
Denote $\partial_B Q_i = \partial Q_i \cap B \times 1$. Similarly to the cores in Figure~\ref{figure:extended_core} we extend $Q_i$ 
to a disk in $W_K$, denoted $\hat Q_i$, using the decomposition~\ref{eq:W_decomposition}
\[
\hat Q_i := \partial_B Q_i \times I \cup \big(\partial_B Q_i \times I \cup Q_i\big) \times 0\,.
\]
Equipped with $\hat Q_i$ we are finally able to define the disks 
\begin{equation}\label{eq:Delta23}
  \Delta^2_{i,2} :=  \bigcup_{\theta \in [4\pi/3,2\pi]} c^\theta_{i,w^-_2} \cup \hat Q_i\,.  
\end{equation}
We emphasize that by construction the disks have boundary on $\Sigma$. Using the compression disk we define the compression bodies $K_j$, for $j=0,1,2$, as follows. It is easy to check that the normal bundle of $\Sigma$ in $W_K$ is trivial. Attach a $3$-dimensional neighborhood of $\Delta^1_{i,j}$ and $\Delta^2_{i,j}$ to $\nu_j\Sigma$, which is a small $3$-dimensional thickening of $\Sigma$ into the direction with angle $2j\pi/3$ with respect to a fixed normal vector field, for all $i=1,\dots,k$ and $j=0,1,2$. We obtain a compression body with higher genus boundary $\Sigma$ and lower genus boundary a surface $F'$. After isotopy which removes the handle of $\Sigma$ in angle $2(j-1)\pi/3$, the surface $F'$ is not contained in the $(2(j-1)\pi/3)$-page. Hence after a further isotopy, $F'$ is only contained in the $(2j\pi/3)$-page. Since the $2$-handle with core $\Delta^2_{i,j}$ cancels one of the two remaining handles of $\Sigma$ there exists a final isotopy $F'$  to $F_{2j\pi/3}$ as required.

\subsection{Trisection of \texorpdfstring{$W_K$}{WK} and  \texorpdfstring{$W$}{W}} 
Let $c \subset D^2$ denote the circle of radius $1/2$. We define a $2$-complex $Z\subset W_K$ 
    \[Z := B\cup \bigcup_{i,j} \Delta^2_{i,j}  \cup \bigcup_{z \in c \cap D_\rho(z_j)} c^{ 2 j\pi/3}_{i,z} \,.\]
We claim that $W_K$ collapses onto $Z$, i.e.\ there exists an isotopy of $W_K$ which takes the complement of a small regular neighborhood of the boundary onto a small regular neighborhood of $Z$. This is done by an isotopy, which preserves the pages. More precisely, for the $0$-page, this  is the isotopy onto the disks $\hat Q_i$, which exists since these are cores of a handle decomposition. We identify $\hat Q_i$ with a quadrilateral $[-1,1] \times [0,1]$ such that $[-1,1] \times 0$ is identified with $c_{i,w_2^-}^{2\pi}$ and $[-1,1] \times 1$ is identified with $c_{\sigma(i),w^+_2}^{0}$ up to changing the orientation. We extend to pages $K_\theta$ for $\theta<\epsilon$ or $\theta > 2\pi-\epsilon$ by concatenating with the isotopy which takes $Q_i$ to $[-1,1] \times 0$ or $[-1,1]\times 1$ respectively and finally extend this to all pages. 

Secondly, we claim that a regular neighborhood of $\bigcup_j K_j$ also collapses onto $Z$. This is true because $\bigcup_j K_j$ collapses onto the complex
\[\Sigma \cup \bigcup_{i,j} \Delta^1_{i,j} \cup \Delta^2_{i,j}\,,\]
and by the fact that a regular neighborhood of this complex is isotopic to a regular neighborhood of $Z$. 

In conclusion, we find an isotopy of $W_K$ which takes a regular neighborhood of $\bigcup_j K_j$ to $W_K\setminus \nu F_\Phi$ as required. This completes step $(3)$.

As explained in the beginning, we complete the compression bodies $K_j$ using the fibers $H_\theta$ of $W_H$ to get
\begin{equation}\label{eq:L}
  L_j := K_j \cup_{F_j} H_j\,,
\end{equation}
with $F_j = F_{2j\pi/3}$ and $H_j=H_{2j\pi/3}$ for $j=0,1,2$. Note that a regular neighborhood of $\bigcup_j L_j$ is given by $W_K$ union a thickening of the three fibers $H_j$. The complement in $W$ of this regular neighborhood consists of three copies of the form $H \times [0,1]$, which is a handle body in the case when $H$ is a handle body and thus $\partial W=\emptyset$ and is as in the definition for a relative trisection in the case $\partial W \neq \emptyset$. This gives the stated (relative) trisection of $W$. 

It remains to compute the statements about the trisection genera. Assume that $\partial W=\emptyset$. The genus of the trisection constructed above is $(\ell+2k,\ell)$. Now we see that $\ell=g$, $\chi(B)-2k=\chi(F)=2-2g$ and $\chi(B)=2|B|-2g(B)$ and thus we can also express the trisection genus as $(3g-2g(B),g)$ as stated in the theorem.

Now suppose that $\partial W \neq \emptyset$. We have that $W$ admits a structure of an open book $\Ob(M,\Phi)$, where $M$ is a sutured $3$-manifold with suture $\partial M  = B \cup_C P$. The number of boundary components $b$ is equal for $\Sigma, F, P$ and $B$ and equals the number of components of $C$. The above-described algorithm gives a relative $(g+2k,g;p,b)$-trisection. We compute $\chi(B)-2k=\chi(F)$ and $\chi(B)=2|B|-2g(B)-b$ as well as $\chi(F)=2-2g-b$. We conclude the formula for the genus of the trisection. 

\begin{figure}
\centering
\def\svgwidth{0.99\columnwidth}
\begingroup%
  \makeatletter%
  \providecommand\color[2][]{%
    \errmessage{(Inkscape) Color is used for the text in Inkscape, but the package 'color.sty' is not loaded}%
    \renewcommand\color[2][]{}%
  }%
  \providecommand\transparent[1]{%
    \errmessage{(Inkscape) Transparency is used (non-zero) for the text in Inkscape, but the package 'transparent.sty' is not loaded}%
    \renewcommand\transparent[1]{}%
  }%
  \providecommand\rotatebox[2]{#2}%
  \newcommand*\fsize{\dimexpr\f@size pt\relax}%
  \newcommand*\lineheight[1]{\fontsize{\fsize}{#1\fsize}\selectfont}%
  \ifx\svgwidth\undefined%
    \setlength{\unitlength}{629.35509389bp}%
    \ifx\svgscale\undefined%
      \relax%
    \else%
      \setlength{\unitlength}{\unitlength * \real{\svgscale}}%
    \fi%
  \else%
    \setlength{\unitlength}{\svgwidth}%
  \fi%
  \global\let\svgwidth\undefined%
  \global\let\svgscale\undefined%
  \makeatother%
  \begin{picture}(1,0.47140827)%
    \lineheight{1}%
    \setlength\tabcolsep{0pt}%
    \put(0,0){\includegraphics[width=\unitlength,page=1]{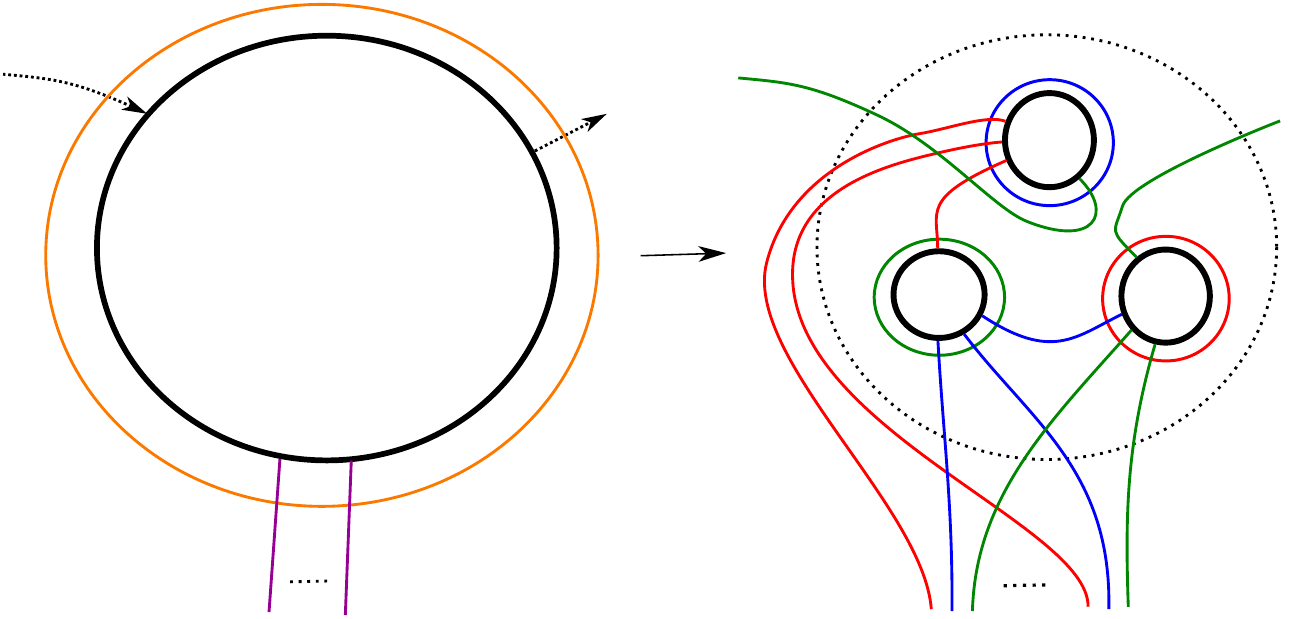}}%
    \put(0.22233217,0.2751708){\color[rgb]{0,0,0}\makebox(0,0)[lt]{\lineheight{1.25}\smash{\begin{tabular}[t]{l}$D_i^+$\end{tabular}}}}%
    \put(0.69605801,0.24121278){\color[rgb]{0,0,0}\makebox(0,0)[lt]{\lineheight{1.25}\smash{\begin{tabular}[t]{l}$D_{i,1}^+$\end{tabular}}}}%
    \put(0.87083935,0.23935233){\color[rgb]{0,0,0}\makebox(0,0)[lt]{\lineheight{1.25}\smash{\begin{tabular}[t]{l}$D_{i,2}^+$\end{tabular}}}}%
    \put(0.78060499,0.3602851){\color[rgb]{0,0,0}\makebox(0,0)[lt]{\lineheight{1.25}\smash{\begin{tabular}[t]{l}$D_{i,0}^+$\end{tabular}}}}%
    \put(0.38791506,0.43605239){\color[rgb]{0.98823529,0.4745098,0}\makebox(0,0)[lt]{\lineheight{1.25}\smash{\begin{tabular}[t]{l}$\epsilon$\end{tabular}}}}%
    \put(0.28279661,0.04534655){\color[rgb]{0.59215686,0,0.58039216}\makebox(0,0)[lt]{\lineheight{1.25}\smash{\begin{tabular}[t]{l}$\delta$\end{tabular}}}}%
  \end{picture}%
\endgroup%

\caption{Left: A Heegaard diagram of the page. The dotted arrow indicates the trace of an elementary handle slide. Right: A trisection diagram of the corresponding open book with common color coding $\alpha$, $\beta$, $\gamma$-curves in red, blue and green respectively.}
	\label{fig:trisect_algo}
\end{figure}

\subsection{Trisection diagram} Since the proof above explicitly constructs the trisection surface $\Sigma$ and handle bodies (resp.\ compression bodies in the case of a relative trisection) given in Equation~\eqref{eq:L} by attaching $2$- and $3$-handles to $\Sigma$, the trisection diagram is directly obtained from a Heegaard diagram of the page by recovering the attaching spheres. We assume without loss of generality that the Heegaard diagram $(F,\delta,\epsilon)$ of $M$ is given in the following  form
\begin{itemize}
     \item each connected component of $\partial_- K$ is represented by a copy of $S^2$ equipped with attaching regions for $1$-handles, and
     \item additional pairs of attaching regions corresponding to the $1$-handles of $K$. The $\e$-curves are required to be parallel to these regions. 
\end{itemize}
First, we replace each attaching region $D_i^\pm$ which has a parallel $\e$-curve with three new attaching regions $D_{i,j}^\pm$ for $j=0,1,2$ (cf.\ Figure~\ref{fig:trisect_algo}). This gives a model of the trisection surface $\Sigma$. To describe the curves $\alpha, \beta$ and $\gamma$ of the trisection diagram $(\Sigma,\alpha,\beta,\gamma)$, we abbreviate $\theta_0=\alpha$, $\theta_1=\beta$ and $\theta_2=\gamma$. For each $j=0,1,2$ and $i=1,\dots,k$ we have a $\theta_j$-curve inside $D^+_i$ running parallel around the $D^+_{i,{j-1}}$-region, corresponding to $\partial \Delta^1_{i,j}$, where we have used the index $j$ modulo $3$. Secondly, we consider the curves corresponding to $\partial \Delta^2_{i,j}$. For each $j=0,1$ and $i=1,\dots,k$ we have a $\theta_j$-curve given by arcs of the circle with radius $\frac 12$ inside $D^\pm_i$ connecting $D^\pm_{i,j}$ to $D^\pm_{i,j+1}$. For $j=2$ we have $\theta_2$-curve given by the attaching sphere of the $2$-handle $b_i$ of Lemma~\ref{lem:icebreaker}. More precisely after isotopy, we assume that the monodromy is given as a product of elementary handle maps and we obtain a $\theta_2$-curve by concatenating the corresponding curves given in Figure~\ref{fig:buildingblocks} and removing potential intersection points using finger moves as explained in Lemma~\ref{lem:icebreaker}. Finally for each $j=0,1,2$ and $\delta$-curve we get a $\theta_j$-curve, which runs parallel to the $\delta$-curve outside $D^\pm_i$ and if the $\delta$-curve passes through $D^\pm_i$ the corresponding $\theta_j$-curve is extended inside $D^\pm_i$ with two arcs to $D^\pm_{i,j}$ which agree under the identification of $D^+_i$ with $D^-_i$ forming a closed curve. If $j=2$ there are potential intersection points with other $\theta_j$-curves previously constructed. We remove these using finger moves as explained in Lemma~\ref{lem:icebreaker}, see also Figure~\ref{fig:trisect_algo}. We remark that the obtained diagram depends on certain choices. However, it is not hard to see that all these trisection diagrams are handle slide equivalent, and thus all represent the same trisection. This completes the proof of Theorem~\ref{thm:A}(1) and~(3).

\begin{figure}
\begin{minipage}{.35\textwidth}
\includegraphics[width=.5\textwidth]{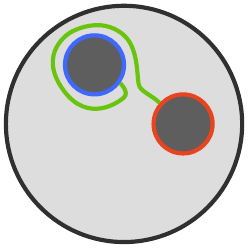}
\end{minipage}
\begin{minipage}{.35\textwidth}
    \includegraphics[width=\textwidth]{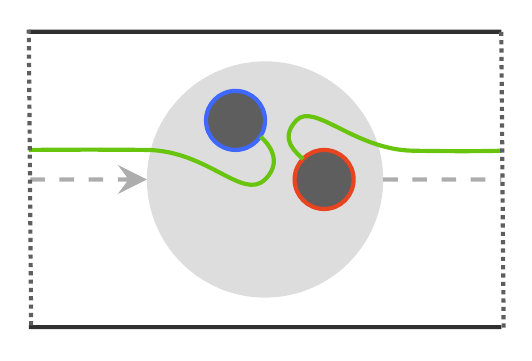}
\end{minipage}
\begin{minipage}{.35\textwidth}
\includegraphics[width=.5\textwidth]{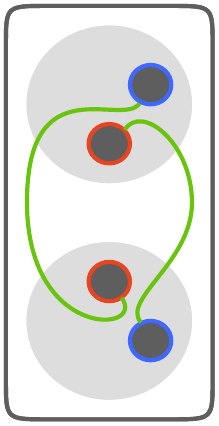}
\end{minipage}
\begin{minipage}{.35\textwidth}
\includegraphics[width=\textwidth]{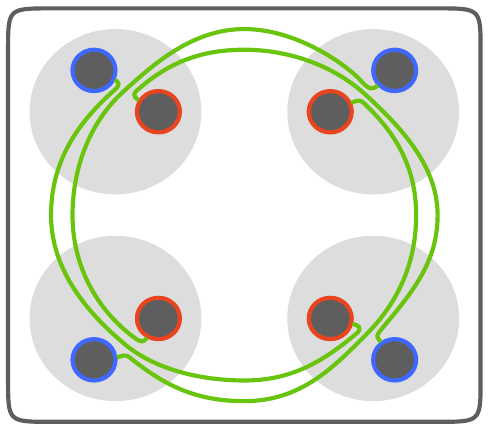}
\end{minipage}
    \caption{The $\gamma$-curves obtained for elementary handle maps. The dotted lines are identified in the picture top right. The black curves bound the disks/annuli along which the (half)-twist is performed. In the pictures below the attaching regions are identified via reflection along a horizontal line. } 
    \label{fig:buildingblocks}
\end{figure}

\subsection{Empty binding}
The case when the binding is empty, i.e.\ when $W$ is an $S^1$-bundle over a $3$-manifold, works similarly. We outline the arguments in the following. Let $W$ be an $S^1$-bundle over a $3$-manifold $M$ with monodromy $\Phi$ which preserves a Heegaard splitting $M=K\cup_F H$. We have the corresponding splitting of $W$ into mapping cylinders
\[W_K := K_\Phi\,, \qquad W_H := H_\Phi\,,\]
which again intersect along the mapping cylinder $F_\Phi$. We fix a handle decomposition of $K$ with exactly one $0$-handle
\[ K=D^3 \cup a_1 \cup \dots \cup a_k\,.\]
We assume without loss of generality that $\Phi$ is the identity on the $0$-handle. We have a corresponding decomposition of $W_K$ given by
\[
W_K = D^3 \times S^1 \cup \left(a_1 \cup \dots \cup a_k \right)\times [0,2\pi]/\sim\,.
\]
We fix an identification of $D^3$ with $S^2 \times [0,1]/\! \sim$ where $(p,0)\sim(q,0)$ for all $p,q \in S^2$ via polar coordinates and another identification $D^3 \cong D^2 \times [-1,1]$. Similarly as above for $z \in D^2$, $i=1,\ldots,k$ and $\theta \in [0,2\pi]$ we define arcs $c^\theta_{i,z}$ and $\bar c_{j,z}$. More precisely we identify $a_i$ with $[-1,1] \times D^2$ and define $p_i^\pm \in S^2$  such that $(p_i^\pm,1)$ is identified with $(\pm 1,z)$, using $D^3 \cong S^2 \times [0,1]/\sim$
\[
c^\theta_{i,z} := \big(\{p_i^-,p_i^+\} \times [1/2,1] \times \theta \big)\cup \big([-1,1] \times z \times \theta\big) \subset \big(D^3 \cup a_i\big) \times \theta\,,
\]
and using $D^3 \cong D^2 \times [-1,1]$ for $j=0,1,2$ 
\[
\bar c_{j,z} := \bigcup_{t\in[0,1]} z \times (t-1/2) \times 2(j+t)\pi/3 \subset D^3 \times S^1\,.
\]
Now define using again $D^3 \cong S^2 \times [0,1]/\sim$ 
\[
B:=\bigcup_{j=0,1,2} S^2 \times 1/2 \times  2j\pi/3  \subset D^3 \times S^1 \,.
\]
Denote the points $p_\pm \in S^2$ such that $(p_\pm,1/2)$ is identified with $(0,\pm 1/2)$ under the identification of  $S^2 \times [0,1]/\sim$ and $D^2 \times [-1,1]$ with $D^3$. Assume without loss of generality that $D_\rho(p_\pm) \times 1/2$ is identified with $D_\rho(0) \times \pm 1$. Further, we assume  that there are no attaching regions near $(p_\pm,1)$. Now define the trisection surface 
\[
\Sigma:= B \setminus \bigcup_j\bigcup_{z \in D_\rho(0)} \partial \bar c_{j,z} \cup \bigcup_{z \in \partial D_\rho(0)} \bar c_{j,z} \cup  \bigcup_{i,j} \bigcup_{z \in D_\rho(0)} \partial c^{2j\pi/3}_{i,z}  \cup \bigcup_{i,j} \bigcup_{z \in \partial D_\rho(0)} c^{2j\pi/3}_{i,z}\,.
\]
Hence $\Sigma$ is given by three disjoint copies of $S^2$ put in different angles and joined together with three $1$-handles with core $\bar c_{j,0}$ and surgered with three $1$-handles for each $1$-handle of $K$. We now define the compression disks. First, we have three disks corresponding to the cocore of the $\bar c_{j,z}$ handles for $j=0,1,2$
\[\bar \Delta_j := D_\rho(0) \times 0 \times (2j+1)\pi/3\,.\]
Then we have the cocore disks $\Delta^1_{i,j}$ of the $1$-handles with core $c^{2j\pi/3}_{i,0}$ given similarly as in Equation~\eqref{eq:Delta1}. We define the final family 
$\Delta^2_{i,j}$ as depicted in Figure~\ref{fig:wurmloch} in the case of trivial monodromy. Note that $\Delta^2_{i,j}$ has boundary on $\Sigma$ and in particular one boundary component is given by $\bar c_{j,z}$ for some particular $z \in D_\rho(0)$. If the monodromy $\Phi$ is non-trivial, we assume without loss of generality that $\Phi$ is the identity on $S^2 \times 1/2$ and correct $\Delta_{i,3}$ using disks $Q_i$ defined similarly as in Equation~\ref{eq:Delta23}. 

\begin{figure}[htbp] 
\centering
\def\svgwidth{0.99\columnwidth}
\begingroup%
  \makeatletter%
  \providecommand\color[2][]{%
    \errmessage{(Inkscape) Color is used for the text in Inkscape, but the package 'color.sty' is not loaded}%
    \renewcommand\color[2][]{}%
  }%
  \providecommand\transparent[1]{%
    \errmessage{(Inkscape) Transparency is used (non-zero) for the text in Inkscape, but the package 'transparent.sty' is not loaded}%
    \renewcommand\transparent[1]{}%
  }%
  \providecommand\rotatebox[2]{#2}%
  \newcommand*\fsize{\dimexpr\f@size pt\relax}%
  \newcommand*\lineheight[1]{\fontsize{\fsize}{#1\fsize}\selectfont}%
  \ifx\svgwidth\undefined%
    \setlength{\unitlength}{583.96269464bp}%
    \ifx\svgscale\undefined%
      \relax%
    \else%
      \setlength{\unitlength}{\unitlength * \real{\svgscale}}%
    \fi%
  \else%
    \setlength{\unitlength}{\svgwidth}%
  \fi%
  \global\let\svgwidth\undefined%
  \global\let\svgscale\undefined%
  \makeatother%
  \begin{picture}(1,0.3371381)%
    \lineheight{1}%
    \setlength\tabcolsep{0pt}%
    \put(0,0){\includegraphics[width=\unitlength,page=1]{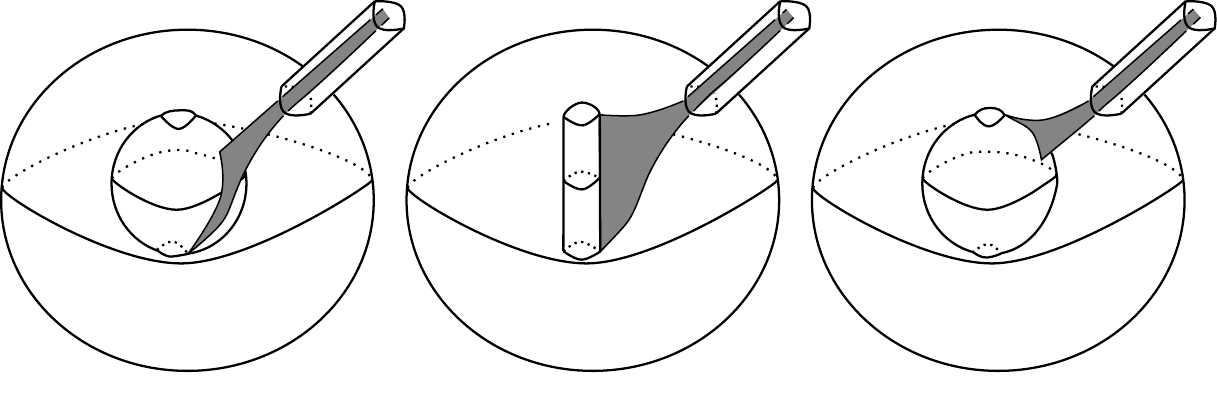}}%
    \put(0.24316511,0.31191277){\color[rgb]{0,0,0}\makebox(0,0)[lt]{\lineheight{0.1}\smash{\begin{tabular}[t]{l}$a_i$\end{tabular}}}}%
    \put(0.57538314,0.31240012){\color[rgb]{0,0,0}\makebox(0,0)[lt]{\lineheight{0.1}\smash{\begin{tabular}[t]{l}$a_i$\end{tabular}}}}%
    \put(0.9138669,0.31103079){\color[rgb]{0,0,0}\makebox(0,0)[lt]{\lineheight{0.1}\smash{\begin{tabular}[t]{l}$a_i$\end{tabular}}}}%
    \put(0.13843823,0.00441734){\color[rgb]{0,0,0}\makebox(0,0)[lt]{\lineheight{0.1}\smash{\begin{tabular}[t]{l}$(i)$\end{tabular}}}}%
    \put(0.47621788,0.00425434){\color[rgb]{0,0,0}\makebox(0,0)[lt]{\lineheight{0.1}\smash{\begin{tabular}[t]{l}$(ii)$\end{tabular}}}}%
    \put(0.80601031,0.00437113){\color[rgb]{0,0,0}\makebox(0,0)[lt]{\lineheight{0.1}\smash{\begin{tabular}[t]{l}$(iii)$\end{tabular}}}}%
    \put(0,0){\includegraphics[width=\unitlength,page=2]{Wurmloch.pdf}}%
  \end{picture}%
\endgroup%

\caption{The compression disk $\Delta^2_{i,j}$. The pictures $(i)$ and $(iii)$ show the parts in the fibre over $2j\pi/3$ and $2(j+1)\pi/3$ respectively. The picture $(ii)$ is a projection of the part which lies in the fibers over the interval $(2j\pi/3,2(j+1)\pi/3)$. Note that one boundary component of $\Delta^2_{i,j}$ is given by $\bar c_{j,z}$}
	\label{fig:wurmloch}
\end{figure}

As before we define compression bodies $K_j$ using these compression disks. Then we see that $\partial_+ K_j=\Sigma$ and $\partial_- K_j=F_{2j\pi/3}$. An argument similar to the open book case shows that the $K_j$ gives a relative $(3k+1,1;\ell+1,0)$ trisection of $W_K$. Next, we add fibers $H_{2j\pi/3}$ to obtain handle bodies (resp. compression bodies) $L_j$ which yield a trisection of $W$ of genus $(3k+1,\ell+1)$ (resp. $(3k+1,\ell+1;p,0)$ in the case with boundary). The same computation as above yields the claimed formulas for the trisection genera given in the theorem. This completes the proof of Theorem~\ref{thm:A}(2) and~(4).

\section{Applications and Examples}\label{sec:ex}

In this section, we will discuss examples of $4$-manifolds with open book decompositions and apply our algorithm from Theorem~\ref{thm:A} to create natural trisection diagrams on these $4$-manifolds. All diagrams in this section are drawn with respect to the color code that the $\alpha$, $\beta$, and $\gamma$-curves are red, blue and green respectively. We start with the standard open books on $D^4$ and~$S^4$.  

\begin{ex} \label{ex:standOB}
We defining the \textit{standard} open book on the closed unit $2$-disk $D^2$ by choosing the binding to be $B=\{0\}\subset D^2$. Then the fibration is given by the angle in polar coordinates
\begin{equation*}
    p_{st}=\theta\colon D^2\setminus\{0\}\rightarrow S^1.
\end{equation*}
The pages of these open books are straight lines from the origin to the boundary $\partial D^2$, i.e.\ line segments $I_\theta$ with constant angle coordinate $\{\theta=\textrm{const}\}$. The monodromy is trivial. For $n\geq 3$ we define the \textit{standard} open book on the closed $n$-disk $D^n$ by crossing the standard open book on $D^2$ with an $(n-2)$-disk. For that, we write $D^n$ as
$D^{n-2}\times D^2$ and define the binding to be $B=D^{n-2}\times\{0\}$. Then the fibration is given by the angular coordinate in the $D^2$-factor, with pages $D^{n-2}\times I_\theta$. The suture on the boundary of the pages is given by an equatorial sphere, i.e.~$C=\partial D^{n-2}\times\{1/2\}$. Thus, abstractly this open book is given by $\operatorname{Ob}(D^{n-1},D^{n-1}_+\cup_{S^{n-2}} D^{n-1}_-,\operatorname{Id})$. 

This induces also the \textit{standard} open book on the boundary sphere $S^{n-1}$, with binding $\partial D^{n-2}=S^{n-3}$, pages $D^{n-2}$ and trivial monodromy. Abstractly this is presented by $\operatorname{Ob}(D^{n-2},\operatorname{Id})$. See Figure~\ref{fig:ob_ball} for an example.

Applying the algorithm from Theorem~\ref{thm:A} to $S^4=\operatorname{Ob}(D^{3},\operatorname{Id})$ readily yields the standard genus $0$-trisection diagram of $S^4$ with trisection surface $S^2$, since the page has no $1$- and $2$-handles. In the same way, $D^4=\operatorname{Ob}(D^{3},D^{3}_+\cup_{S^2} D^{3}_-,\operatorname{Id})$ yields the standard genus-$0$ trisection diagram of $D^4$ with trisection surface $D^2$.
\end{ex}

\begin{figure}[htbp] 
\centering
\def\svgwidth{0,7\columnwidth}
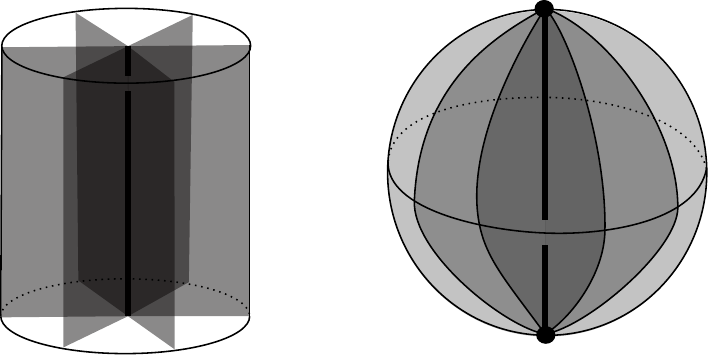
\caption{The standard open book on $D^n$ that induces the standard open book on its boundary $S^{n-1}$.}
	\label{fig:ob_ball}
\end{figure}

\subsection{Surface bundles over the 2-sphere}\label{sec:4dimOB}
Next, we describe simple open books on surface bundles over $S^2$. Trisections of surface bundles over surfaces were also studied before in~\cite{CGP18,CGP18b,CO19,Wi20}. Let $\pi\colon W^4\rightarrow S^2$ be an $F^2$-bundle over $S^2$, for a closed surface $F$. If we split the base $S^2$ into upper and lower hemispheres we get two trivial $F$-bundles over $D^2$. The gluing maps of these bundles are parametrized by $\pi_1(\operatorname{Diff}(F),\operatorname{Id}_F)$. The latter group is known to be isomorphic to $\Z_2$ for $S^2$, to $\Z^2$ for $T^2$, and vanishes for all other closed surfaces~\cite{EE67}. Then the composition
\begin{equation*}
    p\colon W\setminus \pi^{-1}(S^0)\overset{\pi}{\longrightarrow} S^2\setminus S^0\overset{p_{st}}{\longrightarrow} S^1
\end{equation*}
is an open book on $W$ with binding $B=\pi^{-1}(S^0)$, page $M=\pi^{-1}(I)=F\times I$ and the monodromy is given by the element in $\pi_1(\operatorname{Diff}(F),\operatorname{Id}_F)$.

We apply our algorithm from Theorem~\ref{thm:A} to these open books to construct trisection diagrams. First, we describe a Heegaard splitting of the page $F\times I$. For that, we choose a handle decomposition of $F$ with a single $0$-handle, $2g$ $1$-handles (where $g$ is the genus of $F$), and a single $2$-handle. By crossing this handle decomposition with $I$ we get a relative handle decomposition of $F\times I$ with a single $1$-handle, $2g$ $2$-handles, and a single $3$-handle, see Figure~\ref{fig:T2xS2}(a). A collar neighborhood of the boundary together with the $1$-handle describes a compression body $K$, while its complement is a handle body $H$. Their intersection is the Heegaard surface (see Figure~\ref{fig:T2xS2}(b)) given by two copies of $F$ joined via the $1$-handle, i.e.~$F\#F$. To get the trisection surface (if the monodromy is trivial) we replace that $1$-handle with three $1$-handles and add the $(\alpha,\beta,\gamma)$-curves as shown in Figure~\ref{fig:T2xS2}(c). The case of $S^2\times S^2$ is shown in Figure~\ref{fig:S2xS2}(a).

\begin{figure}
\centering
\def\svgwidth{0.9\columnwidth}
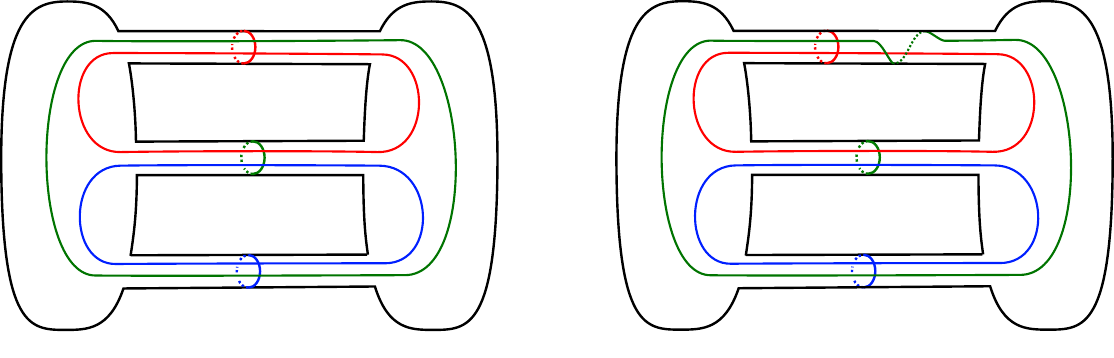
\caption{Trisection diagrams coming from the standard open books of $S^2\times S^2$ in (a) and $S^2\tilde\times S^2$ in (b).}
	\label{fig:S2xS2}
\end{figure}

\begin{figure}
\centering
\def\svgwidth{0.9\columnwidth}
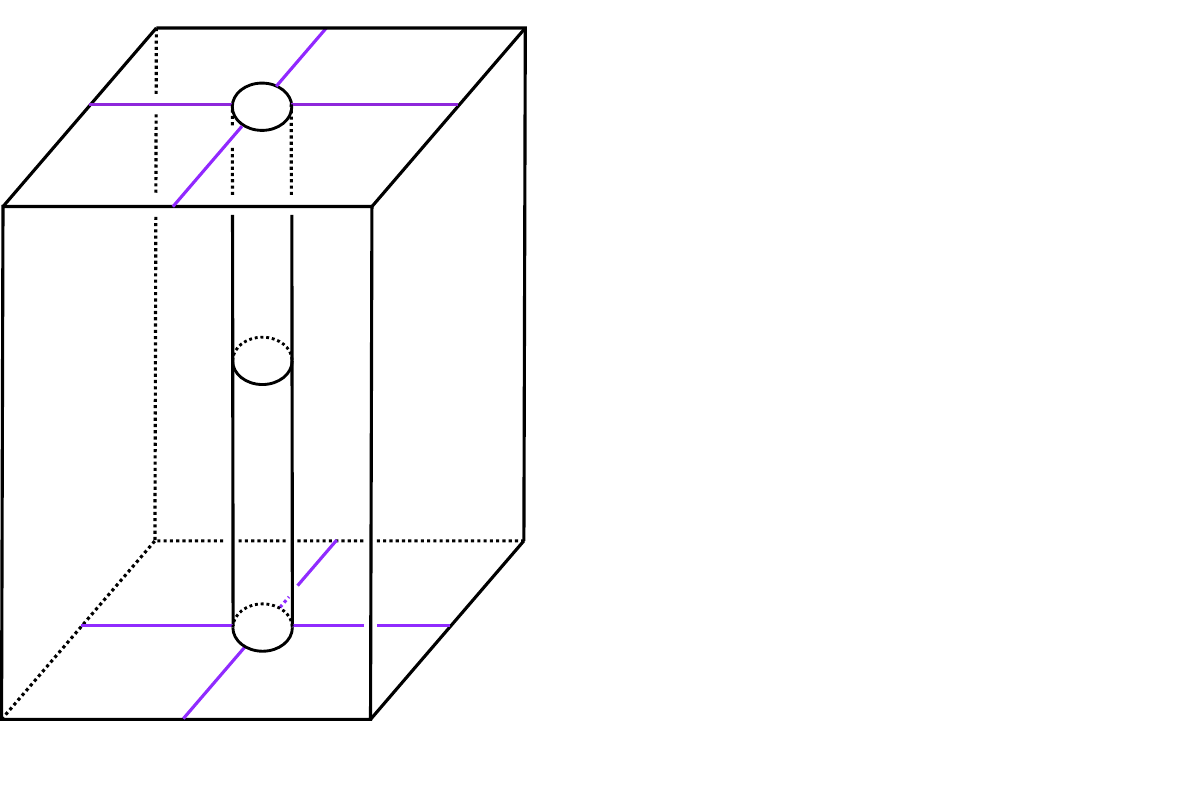
\caption{(a) A handle decomposition of $T^2\times I$. (b) The corresponding Heegaard diagram of $T^2\times I$. (c) The induced trisection diagram of $T^2\times S^2$.}
	\label{fig:T2xS2}
\end{figure}

If the monodromy is non-trivial, we need to modify the $\gamma$-curves. We start by describing the monodromy more concretely. Let $f_t\colon F\rightarrow F$, for $t\in [0,1]$, be a family of diffeomorphisms with $f_0=f_1=\operatorname{Id}_F$. We can see $f_t$ as an element in $\pi_1(\operatorname{Diff}(F),\operatorname{Id}_F)$. Then the monodromy of the corresponding open book is given by 
\begin{align*}
    \Phi\colon F\times I &\longrightarrow F\times I\\
    (p,t)&\longmapsto(f_t(p),t).
\end{align*}
For $F=S^2$ the only non-trivial monodromy is given by a sphere twist $\Phi$ along the central $S^2$ in the page $S^2\times I$. It preserves the Heegaard splitting constructed above and acts as a disk twist along the co-core of the $1$-handle. This shows that the monodromy is given by a single elementary handle map of twisting a foot of a handle. By the algorithm from Theorem~\ref{thm:A} given in the construction of the $2$-handle detailed in the proof of Lemma~\ref{lem:icebreaker} with clarification in Remark~\ref{rem:icebreaker} we  get the $\gamma$ curve from the model $(i)$ from~\ref{fig:awesome_pics}. Thus we get the trisection diagram of $S^2\tilde\times S^2$, the twisted $S^2$-bundle over $S^2$, as shown in Figure~\ref{fig:S2xS2}(b).

The only other non-trivial bundles over $S^2$ are those with fiber $T^2$. Here we know that $\pi_1(\operatorname{Diff}(T^2),\operatorname{Id})$ is isomorphic to $\Z^2$ where $(p,q)\in\Z^2$ corresponds to
\begin{align*}
    f_t\colon T^2 &\longrightarrow T^2 \\
    (\theta_1,\theta_2)&\longmapsto (\theta_1,\theta_2)+t(p,q)
\end{align*}
with $T^2=\R^2/\Z^2$. By choosing a direction orthogonal to $(p,q)$ we see that this bundle actually splits as $S^1$ times the $S^1$-bundle over $S^2$ with Euler number $d$, where $d$ is the least common denominator of $p$ and $q$. After isototpy of $f_t$ the Heegaard splitting constructed above is invariant and $f_t$ restricted to the compression body $K$ is given by sliding one foot of the $1$-handle along the $(p,q)$-curve. Applying the algorithm gives a trisection diagram of the corresponding $T^2$-bundle over $S^2$, see Figure~\ref{fig:twistedT2} for an example.  
\begin{figure} 
\centering
\def\svgwidth{0.8\columnwidth}
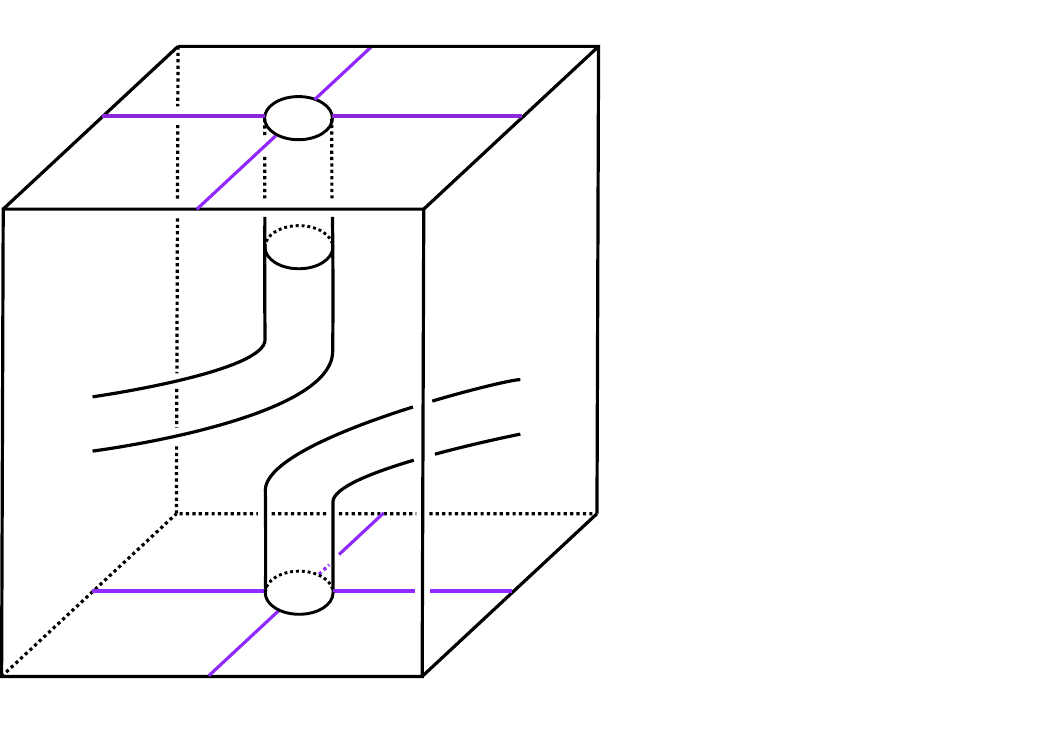
\caption{A trisection diagram of $S^1 \times S^3$ obtained as the $T^2$-bundle over $S^2$ with monodromy given by a twist in $(1,0)$-direction. (a) shows the image of the $1$-handle under the monodromy. (b) is the corresponding trisection diagram. After performing $2$-handle slides we see that the diagram is a stabilization of the standard genus-$1$ diagram of $S^1\times S^3$.}
	\label{fig:twistedT2}
\end{figure}
Thus we have constructed $(2g(F)+2,2g(F))$-trisections on all $F$-bundles over $S^2$. These trisections are minimal since the lower bounds coming from algebraic topology are sharp. 

The same construction also works in the case when $F$ has boundary. We discuss the case of $D^2$-bundles over $S^2$. Then the page is $D^2\times I$ with suture $(D^2\times \partial I) \cup (\partial D^2\times I)$. A Heegaard splitting relative to $D^2\times \partial I$ is obtained by connecting both copies of $D^2\times \partial I$ with a single $1$-handle and no $2$- or $3$-handles, shown in Figure~\ref{fig:disk_bundles}(a). To describe the monodromy we see that $\pi_1(\operatorname{Diff}(D^2))$ is isomorphic to $\Z$ generated by a full rotation. Thus the monodromy of the page $D^2\times I$ is given by an $e$-fold disk twist along the central $D^2\times\{1/2\}$ preserving the above Heegaard splitting. We note that the boundary induces the standard open book of $L(e,1)$ with annulus page and an $e$-fold Dehn twist along its core. From that description, we also see that $e$ represents the Euler class of the bundle. With the same argument as in Figure~\ref{fig:awesome_pics}(i) our algorithm yields the trisection diagram shown in Figure~\ref{fig:disk_bundles}(c).

\begin{figure}[htbp] 
\centering
\def\svgwidth{0.84\columnwidth}
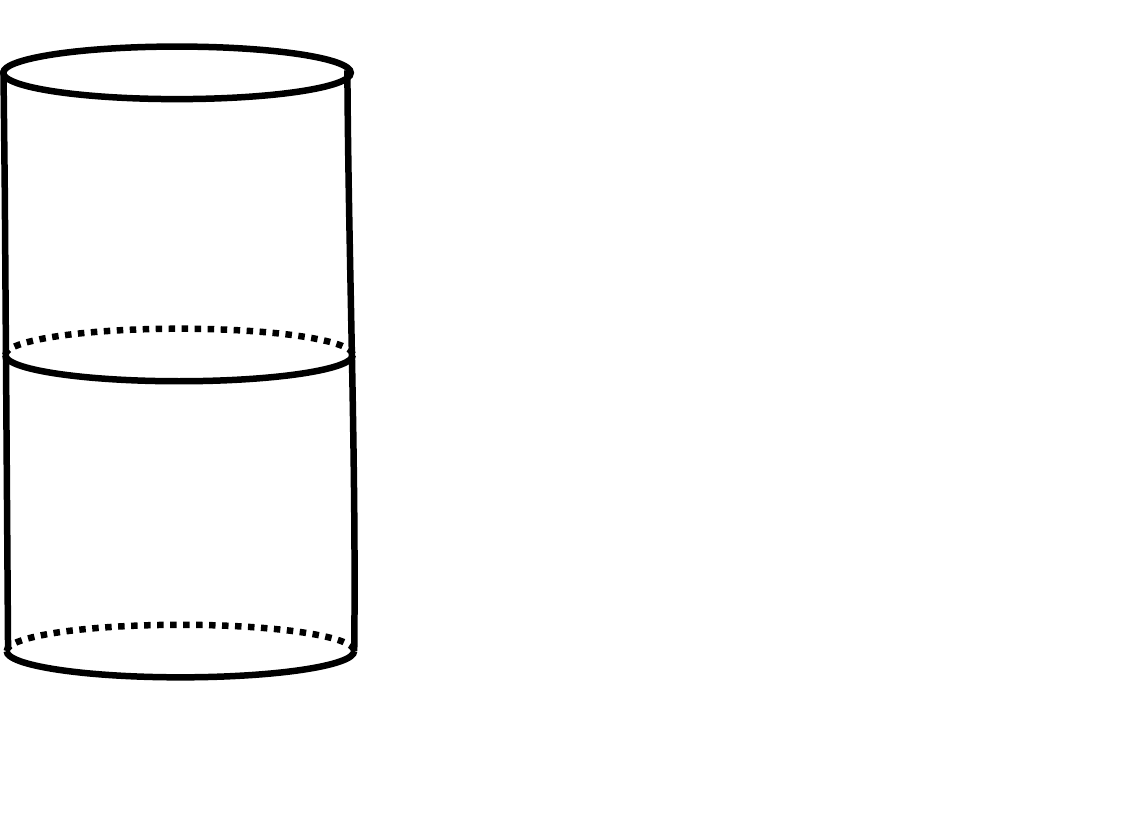
\caption{(a) The page of a disk bundle over $S^2$ is $D^2\times I$, where the suture is given by $D^2\times \partial I \cup \partial D^2\times I$. The monodromy is an $e$-fold disk twist along the shaded central disk. Its Heegaard diagram is shown in (b). From that, we obtain a trisection diagram. In (c) a trisection diagram of the disk bundle over $S^2$ with Euler number $e=2$ is depicted.}
	\label{fig:disk_bundles}
\end{figure}

Next, we will analyze how the trisection diagrams coming from our algorithm are related to other geometric operations. 

\subsection{Connected sums of 4-dimensional open books}
First, we will analyze the connected sum of open books. We recall how to perform connected sums of open books. For that, we remark that any interior point on the binding $B$ of an open book on an $n$-manifold $W$ has a neighborhood $D^n$ in $W$ on which the open book is isotopic to the standard open book from Example~\ref{ex:standOB}. If we have two $4$-manifolds $W_1$, $W_2$ with open book decompositions we get a natural open book decomposition, the \textit{open book connected sum}, on the connected sum $W_1\#W_2$ by performing a connected sum using $n$-disks intersecting the binding as above and gluing the pages of the open books together. If the open books on $W_i$ are presented as abstract open books $\Ob(M_i,\Phi_i)$ then the open book connected sum on $W_1\# W_2$ is given by
\begin{equation*}
    \Ob(M_1\natural M_2,\Phi_1*\Phi_2),
\end{equation*}
where $M_1\natural M_2$ denotes a boundary connected sum of $M_1$ and $M_2$ (which is not unique if the bindings $B_1$ or $B_2$ are not connected) and $\Phi_1*\Phi_2$ denotes the concatenation of the monodromies. The binding of the open book connected sum is given by the connected sum $B_1\#B_2$ of the bindings (again there are several possible choices if the binding is disconnected).

The next result shows that our algorithm is compatible with connected sums.

\begin{prop}\label{prop:con_sum}
    Connected sum of $4$-dimensional open books induce connected sum of trisection diagrams, i.e.~if $(\Sigma_i,\alpha_i,\beta_i,\gamma_i)$, for $i=1,2$, are trisection diagrams obtained via Theorem~\ref{thm:A} from open books $\operatorname{Ob}(M_i,\Phi_i)$, then the trisection diagram obtained via Theorem~\ref{thm:A} from the open book connected sum $\operatorname{Ob}(M_1\natural M_2,\Phi_1*\Phi_2)$ is given by the connected sum $(\Sigma_1,\alpha_1,\beta_1,\gamma_1)\#(\Sigma_2,\alpha_2,\beta_2,\gamma_2)$.
\end{prop}

\begin{proof}
    We perform the boundary connected sum $M_1\natural M_2$ by attaching a $1$-handle to the disjoint union of $M_1$ and $M_2$ such that one foot of the $1$-handle gets attached to $M_1$ and the other to $M_2$. In the dual handle decomposition, this $1$-handle corresponds to a $2$-handle which can be canceled either with the unique $3$-handle of $M_1$ or the unique $3$-handle of $M_2$. After this cancellation, we have a handle decomposition of $M_1\natural M_2$ relative to $B_1\#B_2$ that is away from the reducing sphere given by the old handle decompositions of $M_1$ and $M_2$. Thus the algorithm from Theorem~\ref{thm:A} yields the connected sum of the trisections.
\end{proof}

\subsection{Spuns and twist spuns, revisited}
We remark again that our algorithm generalizes the results of Meier on spuns and twist spuns of $3$-manifolds~\cite{Me18} by seeing the (twist) spun of a closed $3$-manifold $M$ as the open book with page $M\setminus\mathring D^3$ with monodromy the identity (a boundary parallel sphere twist). From a Heegaard diagram of $M$ we obtain directly a Heegaard diagram of $M \setminus \mathring D^3$, which looks identical, by assuming that we take out the $0$-handle without loss of generality. Moreover, for the twist spun, we assume further that the factorization of the sphere twist into elementary handle maps as explained at the end of the proof of Theorem~\ref{thm:handle_slides_generate} is given by twisting the two feet of one $1$-handle each with opposite sign and sliding the other feet of the other handles along concentric circles around these. 

If there is just one $1$-handle in the diagram, the obtained trisection diagram for the twist spun agrees with the trisection diagram of the spun manifold, after pulling the green curves straight, which is possible as both feet are twisted with opposite signs. This recovers the result that the twist spun and spun manifolds are diffeomorphic for manifolds admitting Heegaard splittings of genus one~\cite{Me18}. 

On the other hand, it is known that the twist spun and the spun of a aspherical $3$-manifolds are always non-diffeomorphic~\cite{Pl86}. In general, it is unknown when the spun and the twist spun agree, cf.~\cite{Me18}.

\begin{ex}\label{ex:torus_twists_OB} 
First, we consider the case of spuns of the lens spaces $L(p,q)$. Any Heegaard diagram of $L(p,q)$ gives a trisection diagram, by our algorithm. On the other hand, we see that spuns of the lens spaces $L(p,q)$ have other simple open books, yielding different diagrams. Note that the spun of a lens space is independent of $q$~\cite{Pa77,Pl86,Me18}.

Let $M$ be a solid torus with an open $3$-ball removed. We denote by $\Phi_T$ the twist along the boundary parallel $2$-torus. We define $4$-manifolds $W_p$ by
\begin{equation*} W_p=\operatorname{Ob}(M,\Phi_T^p).
\end{equation*}
\begin{figure}[htbp] 
\centering
\def\svgwidth{0.84\columnwidth}
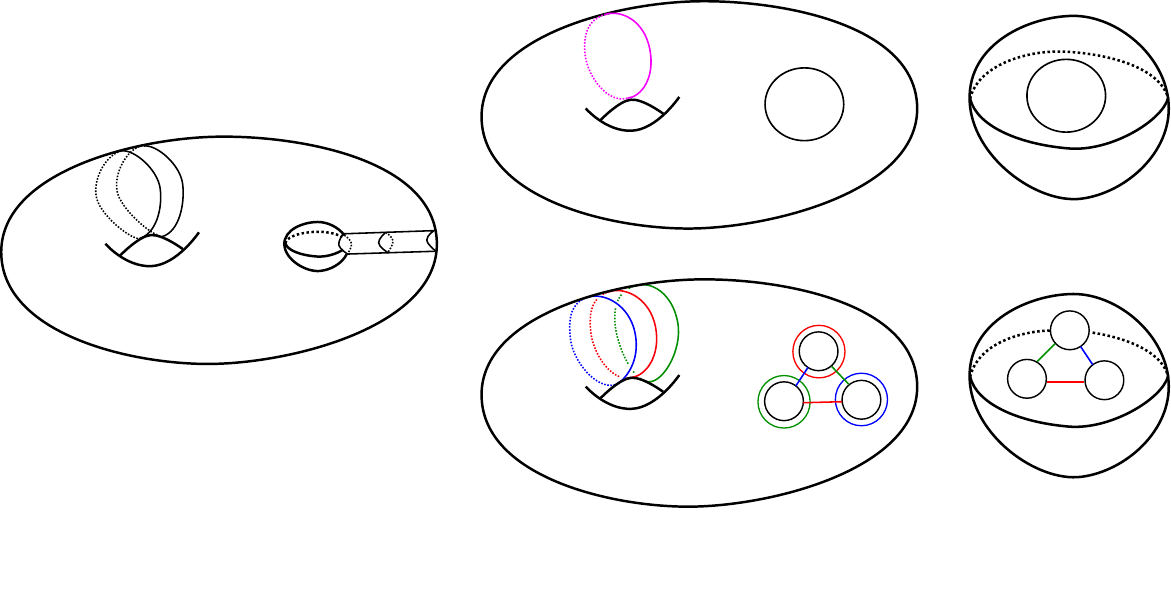
\caption{The handle decomposition of the punctured solid torus $M$ in (a) induces the Heegaard diagram in (b). The open book with page $M$ and trivial monodromy induces the trisection diagram of $S^1\times S^3 \# S^2\times S^2$ in (c).}
	\label{fig:spuns}
\end{figure}
To analyze the diffeomorphism type of $W_p$ we apply the algorithm from Theorem~\ref{thm:A} to this open book. For that we choose a handle decomposition of $M$ given by a single $1$-handle connecting the two boundary components of $M$, a single $2$-handle attached along a meridian of the solid torus, and a unique $3$-handle as shown in Figure~\ref{fig:spuns}(a). The induced Heegaard diagram is shown in Figure~\ref{fig:spuns}(b). This Heegaard splitting is preserved by the monodromy. Figure~\ref{fig:spuns}(c) shows the trisection diagram $W_0$ which is directly identified as the standard trisection diagram of $S^1\times S^3\# S^2\times S^2$. We can see this of course also directly from the open books:
\begin{align*}
    W_0=\operatorname{Ob}(M,\operatorname{Id})&=\operatorname{Ob}(S^1\times D^2\natural S^2\times I,\operatorname{Id})\\
    &=\operatorname{Ob}(S^1\times D^2,\operatorname{Id})\# \operatorname{Ob}(S^2\times I,\operatorname{Id})\\
    &=S^1\times\operatorname{Ob}( D^2,\operatorname{Id})\# S^2\times\operatorname{Ob}( I,\operatorname{Id})\\
    &=S^1\times S^3 \# S^2\times S^2.
\end{align*}
For $p=1$ the trisection diagram is shown in Figure~\ref{fig:stab} where we also show that it is handle slide equivalent to the stabilization of the standard genus $0$-trisection of $S^4$. Similarly, we can show that $W_p$ is the spun of a lens space $L(p,q)$, see for example Figure~\ref{fig:w2} for the case $p=2$. This can also be seen via explicit handle calculus~\cite{Hs23}. Observe that the manifolds $W_p$ are pairwise non-diffeomorphic which implies that $\Phi_T$ seen as an element in the mapping class group of $M$ has infinite order.
\end{ex}

\begin{figure} 
\centering
\def\svgwidth{0.84\columnwidth}
\begingroup%
  \makeatletter%
  \providecommand\color[2][]{%
    \errmessage{(Inkscape) Color is used for the text in Inkscape, but the package 'color.sty' is not loaded}%
    \renewcommand\color[2][]{}%
  }%
  \providecommand\transparent[1]{%
    \errmessage{(Inkscape) Transparency is used (non-zero) for the text in Inkscape, but the package 'transparent.sty' is not loaded}%
    \renewcommand\transparent[1]{}%
  }%
  \providecommand\rotatebox[2]{#2}%
  \newcommand*\fsize{\dimexpr\f@size pt\relax}%
  \newcommand*\lineheight[1]{\fontsize{\fsize}{#1\fsize}\selectfont}%
  \ifx\svgwidth\undefined%
    \setlength{\unitlength}{541.62355438bp}%
    \ifx\svgscale\undefined%
      \relax%
    \else%
      \setlength{\unitlength}{\unitlength * \real{\svgscale}}%
    \fi%
  \else%
    \setlength{\unitlength}{\svgwidth}%
  \fi%
  \global\let\svgwidth\undefined%
  \global\let\svgscale\undefined%
  \makeatother%
  \begin{picture}(1,0.42708369)%
    \lineheight{1}%
    \setlength\tabcolsep{0pt}%
    \put(0,0){\includegraphics[width=\unitlength,page=1]{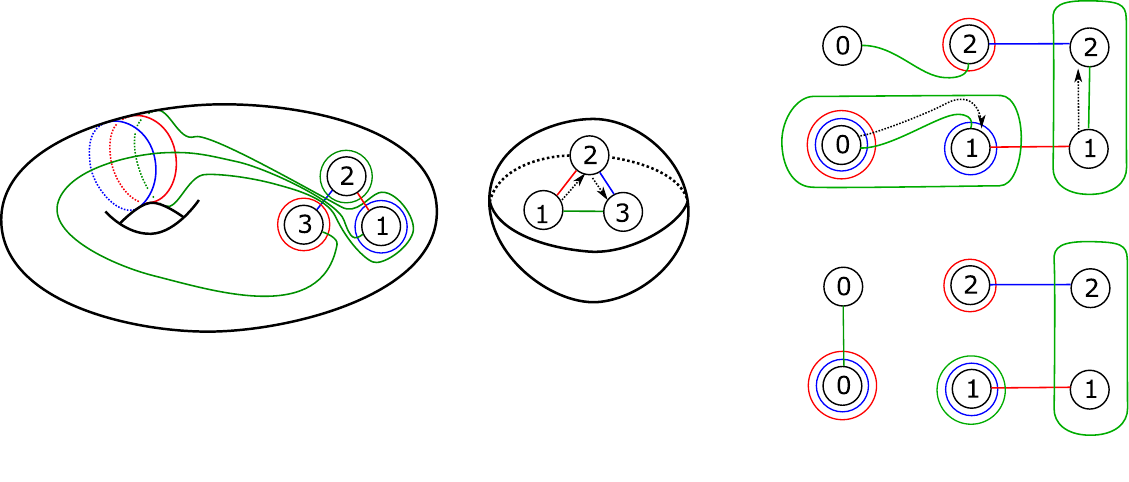}}%
    \put(0.25352443,0.00450802){\color[rgb]{0,0,0}\makebox(0,0)[lt]{\lineheight{1.25}\smash{\begin{tabular}[t]{l}$(a)$\end{tabular}}}}%
    \put(0.86479155,0.22805561){\color[rgb]{0,0,0}\makebox(0,0)[lt]{\lineheight{1.25}\smash{\begin{tabular}[t]{l}$(b)$\end{tabular}}}}%
    \put(0.86479127,0.00334279){\color[rgb]{0,0,0}\makebox(0,0)[lt]{\lineheight{1.25}\smash{\begin{tabular}[t]{l}$(c)$\end{tabular}}}}%
  \end{picture}%
\endgroup%

\caption{By performing a twist along the boundary parallel torus in $M$ we get the trisection diagram of $W_1$ shown in (a). We simplify that diagram by sliding the handles with index $1$ and $2$ as indicated with the dashed arrows and obtain the planar diagram in (b). By further sliding the handle with index $0$ along the dashed line we get the trisection diagram of $W_1$ in (c). (Here we observe that the $0$-indexed $1$-handle as belt spheres in red and blue and thus we can perform $2$-handle slides of the red and blue curves to push the attaching region of the $1$-handle through every blue or red curve.) The resulting diagram destabilizes to the standard genus-$0$ trisection diagram of $S^4$.}
	\label{fig:stab}
\end{figure}

\begin{figure}
    \centering
    \includegraphics[width=.9\textwidth]{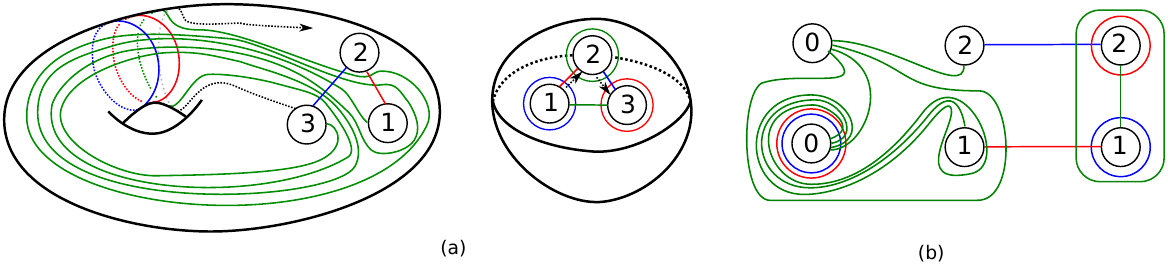}
    \caption{Trisection diagram for $W_2$ shown in $(a)$. As before the dashed arrows indicate the handle slides used to obtain the simplified planar diagram in $(b)$.}
    \label{fig:w2}
\end{figure}

\begin{ex}\label{ex:twistspun} Secondly, we consider the twist spun manifold of the Poincar\'e homology sphere $M$. The Poincar\'e homology sphere has a Heegaard diagram of genus two, given in Figure~\ref{fig:Poincaresphere} (cf.\ \cite[Figure 10]{Me18}). In Figure~\ref{fig:twistspun}, we give the trisection diagram resulting from our algorithm for the twist spun of $M$. 
\end{ex}

\begin{rem}
    We remark that the trisection diagrams of the twist spuns we get from our algorithm do not agree with the trisection diagrams in~\cite{Me18}. In fact, \cite{Me18} has a small mistake in the presentation of the trisection diagrams of the twist spuns, which was later corrected in~\cite{GM22}.
\end{rem}

\begin{figure}
    \centering
    \includegraphics[width=.5\textwidth]{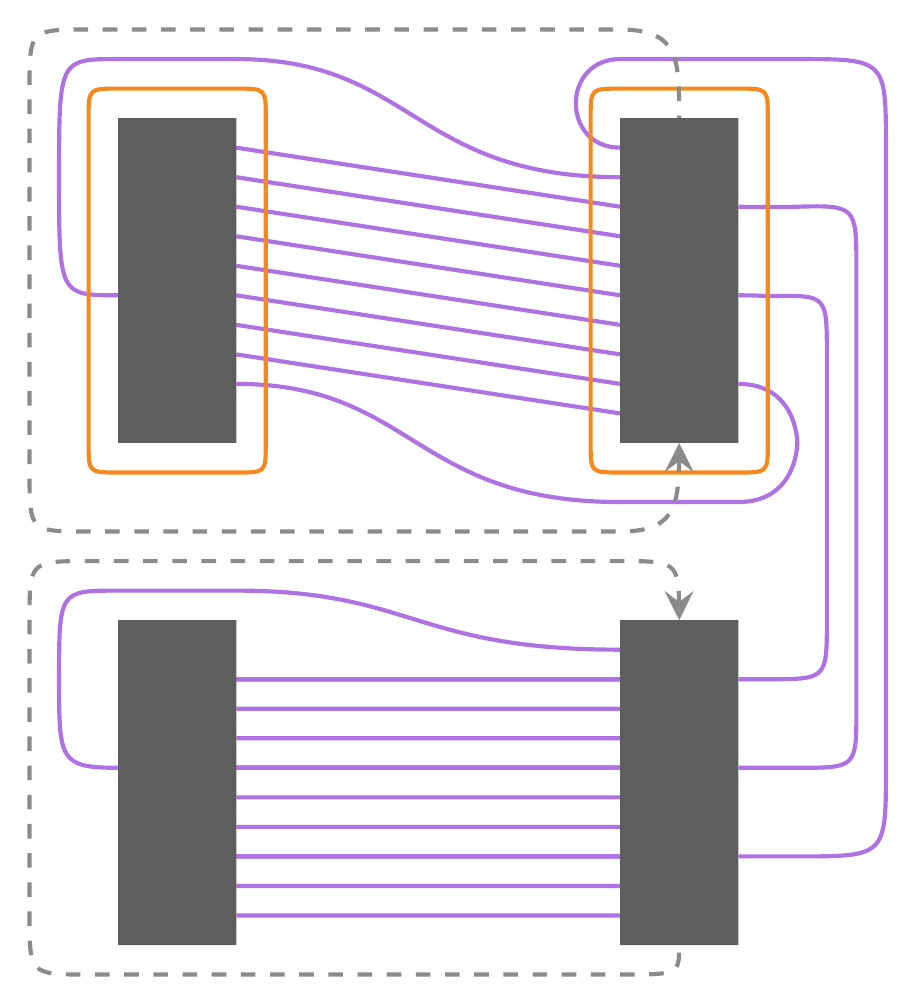}
    \caption{A Heegaard diagram of the Poincar\'e homology sphere. The shaded regions are identified via reflection along a horizontal line passing between the shaded regions. The dotted arrows indicate the monodromy for the twist spun given by sliding the feet of the handles. }
    \label{fig:Poincaresphere}
\end{figure}

\begin{figure}
    \centering
    \includegraphics[width=.75\textwidth]{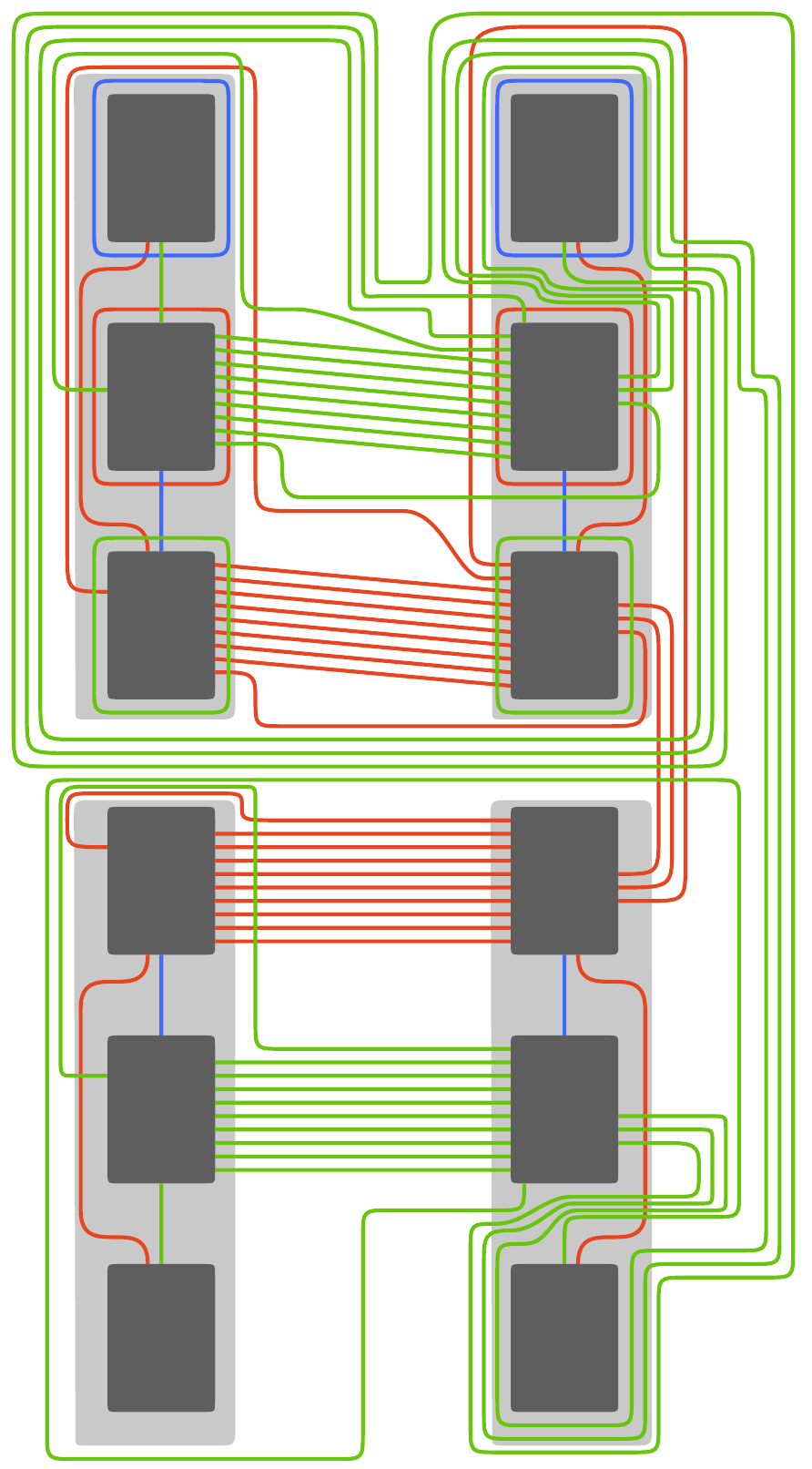}
    \caption{A trisection diagram of the twist spun of the Poincar\'e homology sphere obtained from our algorithm. The dark-shaded regions are identified by reflection along a horizontal line passing between the light-shaded regions. The two blue curves coming from $\delta$-curves are not drawn as they are completely parallel to the long red curves. }
    \label{fig:twistspun}
\end{figure}

\subsection{Stabilizations of 4-dimensional open books}
Next, we will show that stabilizing a $4$-dimensional open book corresponds to stabilizing the trisection. We recall the stabilization operation for $4$-dimensional open books. Let $S^1\times D^2\setminus \mathring D^3$ be a solid torus with a small open $3$-ball removed. We denote by $\Phi_T$ a torus twist along a boundary parallel $2$-torus. In Example~\ref{ex:torus_twists_OB} we have shown that $\Ob(S^1\times D^2\setminus \mathring D^3,\Phi_T)$ is an open book of $S^4$. We define the \textit{stabilization} of an abstract open book $\Ob(M,\Phi)$ on a $4$-manifold to be the open book connected sum of $\Ob(M,\Phi)$ with $\Ob(S^1\times D^2\setminus \mathring D^3,\Phi_T)$. Since $\Ob(S^1\times D^2\setminus \mathring D^3,\Phi_T)$ is an open book of $S^4$ this operation does not change the underlying $4$-manifold. The next result says that $4$-dimensional stabilization of open book is compatible with trisections.

\begin{prop}
Stabilizing a $4$-dimensional open book corresponds to stabilizing the trisection diagrams, i.e.~if $(\Sigma,\alpha,\beta,\gamma)$ is the trisection diagram obtained via Theorem~\ref{thm:A} from an open book $\operatorname{Ob}(M,\Phi)$, then the trisection diagram obtained via Theorem~\ref{thm:A} from the stabilized open book $\operatorname{Ob}(M\natural S^1\times D^2\setminus \mathring D^3,\Phi*\Phi_T)$ is given by the stabilization of $(\Sigma,\alpha,\beta,\gamma)$.
\end{prop}

\begin{proof}
We have shown in Example~\ref{ex:torus_twists_OB} that $\Ob(S^1\times D^2\setminus \mathring D^3,\Phi_T)$ yields the stabilized trisection diagram of $S^4$. Proposition~\ref{prop:con_sum} then implies that if we apply the algorithm from Theorem~\ref{thm:A} to $\operatorname{Ob}(M\natural S^1\times D^2\setminus \mathring D^3,\Phi*\Phi_T)$ we get a stabilization of $(\Sigma,\alpha,\beta,\gamma)$.
\end{proof}

\subsection{Stabilizing 3-dimensional open books}
On the other hand, we can consider a $3$-dimensional open book $\Ob(F^2,\Phi)$ of a closed $3$-manifold $N$, where the page is a surface $F$. Then we get an induced open book on $S^1\times N$ by $\Ob(S^1\times F,\operatorname{Id}_{S^1}\times \Phi)$. 

\begin{ex} We consider $N=S^3$ with standard open book $\Ob(D^2,\operatorname{Id})$. This induces an open book on $S^1\times S^3$ given by $\Ob(S^1\times D^2,\operatorname{Id})$. $S^1\times D^2$ has a handle decomposition consisting of a single $2$-handle and a single $3$-handle. And thus the algorithm from Theorem~\ref{thm:A} readily produces the standard genus-$1$ trisection diagram of $S^1\times S^3$.
\end{ex}

We can also perform a $3$-dimensional stabilization of the open book on $N^3$ and perform the same construction. Note that the resulting open book is not a $4$-dimensional stabilization since the topologies of the pages differ. 

\begin{ex}
We consider $N=S^3$ with the stabilized open book $\Ob(S^1\times I,\Phi)$ where $\Phi$ denotes a Dehn twist along the core curve of the annulus. This induces an open book on $S^1\times S^3$ given by $\Ob(S^1\times S^1\times I,\operatorname{Id}\times \Phi)$. In fact, we have already analyzed this open book and its trisection diagram in Section~\ref{sec:4dimOB} and Figure~\ref{fig:twistedT2}. In particular, one can see that it induces a trisection diagram of the stabilization of the standard trisection of $S^1\times S^3$.
\end{ex}

A similar construction works on more general $S^1$-bundles over $3$-manifolds by first constructing open books coming from open books on $M$ and then constructing trisection diagrams.



\begin{thebibliography} {11111111} 
\bibitem[BS18]{BS18} \textsc{\.{I}. Baykur and O. Saeki}, Simplified broken Lefschetz fibrations and trisections of $4$-manifolds, \textit{Proc. Natl. Acad. Sci. USA} \textbf{115} (2018), 10894--10900.

\bibitem[BS21]{BS21} \textsc{\.{I}. Baykur and O. Saeki}, Simplifying indefinite fibrations on $4$-manifolds, \textit{Trans. Amer. Math. Soc.} \textbf{376} (2023), 3011--3062.

\bibitem[BHRT18]{BHRT18} \textsc{M. Bell, J. Hass, J. Rubinstein, and S. Tillmann}, Computing trisections of 4-manifolds, \textit{Proc. Natl. Acad. Sci. USA} \textbf{115} (2018), 10901--10907.

\bibitem[CGP18a]{CGP18} \textsc{N. Castro, D. Gay, and J. Pinz\'{o}n-Caicedo}, Trisections of $4$-manifolds with boundary, \textit{Proc. Natl. Acad. Sci. USA} \textbf{115} (2018), 10861--10868.

\bibitem[CGP18b]{CGP18b} \textsc{N. Castro, D. Gay, and J. Pinz\'{o}n-Caicedo}, Diagrams for relative trisections, \textit{Pacific J. Math.} \textbf{294} (2018), 275--305.
 
\bibitem[CO19]{CO19} \textsc{N. Castro B. and Ozbagci}, Trisections of $4$-manifolds via Lefschetz fibrations, \textit{Math. Res. Lett.} \textbf{26} (2019), 383--420.

\bibitem[CPV18]{CPV18} \textsc{V. Colin, F. Presas, and T. Vogel}, Notes on open book decompositions for Engel structures, \textit{Algebr. Geom. Topol.} \textbf{18} (2018), 4275--4303.

\bibitem[Di23]{Di23} \textsc{R. Dissler}, Relative trisections of fiber bundles over the circle, \texttt{arXiv:2304.09300}.

\bibitem[EE67]{EE67} \textsc{C. Earle and J. Eells}, The diffeomorphism group of a compact {R}iemann surface, \textit{Bull. Amer. Math. Soc.} \textbf{73} (1967), 557--559.

\bibitem[Et06]{Et06} \textsc{J. Etnyre}, Lectures on open book decompositions and contact structures, in: \textit{Floer homology, gauge theory, and low-dimensional topology}, Clay Math. Proc. \textbf{5} (Amer. Math. Soc., Providence, RI, 2006), 103--141.

\bibitem[FKSZ18]{FKSZ18} \textsc{P. Feller, M. Klug, T. Schirmer, and D. Zemke}, Calculating the homology and intersection form of a $4$-manifold from a trisection diagram, \textit{Proc. Natl. Acad. Sci. USA} \textbf{115} (2018), 10869--10874.

\bibitem[FM20]{FM20} \textsc{V. Florens, and D. Moussard}, Torsions and intersection forms of 4-manifolds from trisection diagrams, \textit{Canadian Journal of Mathematics} \textbf{1} (2020), 1--23.

\bibitem[GK15]{GK15} \textsc{D. Gay and R. Kirby}, Indefinite Morse 2–functions: Broken fibrations and generalizations, \textit{Geom. Topol.} \textbf{19} (2015), 2465--2534.

\bibitem[GK16]{GK16} \textsc{D. Gay and R. Kirby}, Trisecting $4$-manifolds, \textit{Geom. Topol.} \textbf{20} (2016), 3097--3132.

\bibitem[GM22]{GM22} \textsc{D. Gay and J. Meier}, Doubly pointed trisection diagrams and surgery on $2$-knots, \textit{Math. Proc. Cambridge Philos. Soc.} \textbf{172} (2022), 163--195.

\bibitem[Gi02]{Gi02} \textsc{E. Giroux}, G\'{e}om\'{e}trie de contact: de la dimension trois vers les dimensions sup\'{e}rieures, in: \textit{Proceedings of the International Congress of Mathematicians II} (Beijing, 2002) 405--414.

\bibitem[Ha20]{Ha17} \textsc{K. Hayano}, On diagrams of simplified trisections and mapping class groups, \textit{Osaka J. Math.} \textbf{57} (2020), 17--37.

\bibitem[He20]{He20} \textsc{S. Hensel}, A primer on handlebody groups, in: \textit{Handbook of group actions V} (Somerville, 2020) 143--177.

\bibitem[Hs23]{Hs23} \textsc{C.-S. Hsueh}, Kirby diagrams of $4$-dimensional open books, \texttt{arXiv:2306.16942}.

\bibitem[Is21]{Is18} \textsc{G. Islambouli}, Nielsen equivalence and trisections of 4-manifolds, \textit{Geom. Dedicata} \textbf{214} (2021), 303--317.

\bibitem[CIMT22]{CIMT22} \text{N. A. Castro, G. Islambouli,M. Miller, and M. Tomova}, The relative L-invariant of a compact 4-manifold, \textit{Pacific journal of mathematics}  \textbf{315(2)} (2022), 305-346.

\bibitem[It72]{It72} \textsc{T. Itiro}, Spinnable structures on differentiable manifolds, \textit{Proc. Japan Acad.} \textbf{48} (1972), 293--296.

\bibitem[It73]{It73} \textsc{T. Itiro}, Foliations and spinnable structures on manifolds, \textit{Ann. Inst. Fourier (Grenoble)} \textbf{23} (1973), 197--214.

\bibitem[KS23]{KS23} \textsc{M. Kegel and F. Schmäschke}, The stable existence of open books on $4$-manifolds, in preparation, 2024.

\bibitem[Ke22]{Ke22} \textsc{W. Kepplinger}, An algorithm taking Kirby diagrams to trisection diagrams, \textit{Pacific J. Math.} \textbf{318} (2022), 109--126.

\bibitem[Ko17]{Ko17} \textsc{D. Koenig}, Trisections of $3$-manifold bundles over $S^1$, \textit{Algebr. Geom. Topol.} \textbf{21} (2021), 2677--2702.

\bibitem[Mc06]{Mc06} \textsc{D. McCullough}, Homeomorphisms which are Dehn twists on the boundary,\textit{ Algebr. Geom. Topol.} \textbf{6} (2006), 1331--1340.

\bibitem[MM86]{MM86} \textsc{D. McCullough and A. Miller}, Homeomorphisms of $3$-manifolds with compressible boundary, \textit{Mem. Amer. Math. Soc.}, \textbf{344} (1986), 1--100.

\bibitem[Me18]{Me18} \text{J. Meier}, Trisections and spun $4$-manifolds, \textit{Math. Res. Lett.} \textbf{25} (2018), 1497--1524.

\bibitem[Mi21]{Mi21} \text{M. Miller}, Extending fibrations on knot complements to ribbon disk complements, \textit{Geom. Topol.} \textbf{25} (2021), 1479--1550.

\bibitem[Mi68]{Mi68} \textsc{J. Milnor}, \textit{Singular points of complex hypersurfaces}, Annals of Mathematics Studies \textbf{61}, Princeton University Press, Princeton (1968).

\bibitem[Oe02]{Oe02} \textsc{U. Oertel}, Automorphisms of three-dimensional handlebodies,  \textit{Topology} \textbf{41} (2002), 363--410.

\bibitem[Pa77]{Pa77} \textsc{P. Pao}, The topological structure of $4$-manifolds with effective torus actions I,  \textit{Trans. Amer. Math. Soc.} \textbf{227} (1977), 279--317.

\bibitem[Pl86]{Pl86} \textsc{S. Plotnick}, Equivariant intersection forms, knots in $S^4$, and rotations in $2$-spheres,  \textit{Trans. Amer. Math. Soc.} \textbf{296} (1986), 543--575.

\bibitem[Qu79]{Qu79} \textsc{F. Quinn}, Open book decompositions, and the bordism of automorphisms, \textit{Topology} \textbf{18} (1979), 55--73.

\bibitem[Re33]{Re33} \textsc{K. Reidemeister}, Zur dreidimensionalen Topologie, \textit{Abh. Math. Sem. Univ. Hamburg} \textbf{9} (1933), 189--194.

\bibitem[Ru90]{Ru90} \textsc{D. Ruberman}, Seifert surfaces of knots in $S^4$, \textit{Pac. J. Math.} \textbf{145} (1990), 97--116.

\bibitem[Si33]{Si33} \textsc{J. Singer}, Three-dimensional manifolds and their Heegaard diagrams, \textit{Trans. Amer. Math. Soc.} \textbf{35} (1933), 88--111.

\bibitem[Ta21]{Ta21} \textsc{H. Tanimoto}, Homology of relative trisection and its application, \textit{J. Knot Theory Ramifications} \textbf{32} (2023), 2350024, 18.

\bibitem[Ta20]{Ta20} \textsc{S. Taylor}, Equivariant Heegaard genus of reducible 3-manifolds, \textit{Math. Proc. Cambridge Philos. Soc.} \textbf{175} (2023), 51--87.

\bibitem[Su77]{Su77} \textsc{S. Suzuki}, On Homeomorphisms of a 3-Dimensional Handlebody, \textit{Canadian Journal of Mathematics}  \textbf{29(1)} (1977), 111--124. 

\bibitem[Wa98]{Wa98} \textsc{B. Wajnryb}, Mapping class group of a handlebody, \textit{Fundamenta Mathematicae} \textbf{158} (1998), 195--228.

\bibitem[Wa68]{Wa68} \textsc{F. Waldhausen}, On Irreduceible $3$-Manifolds Which are Sufficiently Large, \textit{Annals of Mathematics}, \textbf{87} (1968), 56--88.

\bibitem[Wi20]{Wi20} \textsc{M. Williams}, Trisections of Flat Surface Bundles over Surfaces, \textit{Ph.D. thesis}, University of Nebraska -- Lincoln (2020). 

\end{thebibliography}
\end{document}